\documentclass[]{article}

\usepackage[utf8]{inputenc}
\usepackage[T1]{fontenc}
\usepackage{lmodern}
\usepackage{amsmath, amsthm, amscd, amsfonts}
\usepackage{amsfonts,epsfig, graphics, hyperref,cleveref, amssymb}
\usepackage{mathrsfs,eufrak}
\usepackage{enumitem}
\usepackage{authblk}
\usepackage[a4paper, total={6in, 9in}]{geometry}
\newtheorem{thm}{Theorem}[section]
\newtheorem{Theorem}{Theorem}[section]
\newtheorem{Definition}{Definition}[section]

\newtheorem{example}[thm]{Example}

\newtheorem{remark}[thm]{Remark}
\newtheorem{lemma}[thm]{Lemma}

\newtheorem{corollary}[thm]{Corollary}
\usepackage{hyperref}
\usepackage{cleveref}
\usepackage{mathrsfs} 
\newcommand{\vC}{\v{C}}
\newcommand{\Co}{\mathfrak{Co}}
\newcommand{\fcs}{fuzzy closure space}
\newcommand{\intt}{\operatorname{int}}
\newcommand{\fts}{fuzzy topological space}

\title{\textbf{Separation Axioms in Fuzzy Closure Spaces}}

\author[1]{Albin James \thanks{Corresponding author: albinjames5@gmail.com, albinjp@cusat.ac.in}}
\author[2]{T. P. Johnson \thanks{tpjohnson@cusat.ac.in} }
\affil[1]{Department of Mathematics, Cochin University of Science and Technology, Cochin-22, Kerala , India}
\affil[2]{Applied Sciences and Humanities Division, School of Engineering,  Cochin University of Science and Technology, Cochin-22, Kerala, India}
\date{}

\begin{document}
	\maketitle

	\thispagestyle{empty}

	\begin{abstract}
		Fuzzy closure spaces are an extension of classical closure spaces in topology, where the concept of closure is defined in terms of fuzzy sets. This article introduces interior operators and neighborhood systems in fuzzy closure spaces. Using that, we have redefined \v{C}F-continuity. Separation axioms such as $\v{C}FT_0$, $\v{C}FT_1$,\text{ and } $\v{C}FT_2$, \v{C}F-Urysohn, \v{C}F-regular, and \v{C}F-normal in fuzzy closure spaces are introduced using these neighborhood systems. Additive, productive, hereditary, and other properties of these axioms have been observed. Relationships between these axioms are also investigated.
		
		\vspace*{.5cm}
	\end{abstract}

			\section{Introduction}
			
			The concept of a topological space is generally introduced in terms of the axioms of the open sets. However, there are different methods to describe a topology on the set $X$ are often used in terms of neighborhood systems, the family of closed sets, the closure operator, the interior operator, etc.  To each topological space $(X, \tau)$ we can associate a closure operator $c:P(X) \rightarrow P(X)$, where $c(A)$ is the smallest closed set which contains $A$ for each $A \subseteq X$. This closure operator have the following properties, (i) $c(\phi)=\phi$ (ii) $A \subseteq c(A)$ (iii) $c$ commutes with finite unions (iv) $c$ is idempotent. Such a closure operator is known as a topological closure operator or Kuratowski's closure operator. In 1966, Edward \v{C}ech \cite{CECH}  introduced \v{C}ech closure operators by weakening the idempotent law of topological closure operators. A set $X$ together with a \v{C}ech closure operator is called a \v{C}ech closure space. Thus \v{C}ech closure space can be considered as a generalization of topological spaces. Subsequently, Birkoff \cite{BIRK} defined closure operator in another way by weakening the finite additivity property of the topological closure operator.
			
			By extending the idea of crisp sets, Zadeh\cite{ZADE} established the concept of fuzzy sets in 1965. Chang\cite{CHAN} later gave fuzzy sets a topological structure. Later, the concept of fuzzy closure operators on a set $X$ was introduced separately by Mashhour, Ghanim \cite{MASH} and Srivastava et al.\cite{SRIV}. The definition of Mashhour and Ghanim is an analogue of \v{C}ech closure operators. Srivastava et al. \cite{SRIV} have introduced fuzzy closure operators as an analogue of Birkoff closure operator\cite{BIRK}. Srivastava et al. \cite{REKH} introduced $T_0$ and $T_1$ separation properties in their fuzzy closure spaces. In their definition of fuzzy closure spaces, most of the properties will coincide with the fuzzy topology associated with it. Even the continuous maps of fuzzy closure spaces will be the same as that of it's associated fuzzy topological spaces. So we are working on fuzzy closure spaces introduced by Mashhour and Ghanim \cite{MASH}.
			\par  In 1985, Mashhour and Ghanim\cite{MASH} have introduced the fuzzy closure spaces and extended classical concepts such as subspaces, sums, and products. They have introduced \v{C}F-continuity, compactness, and regularity and studied its properties. Also, they have investigated the interaction between fuzzy proximity spaces and fuzzy closure spaces. Johnson \cite{JOHN} studied completely homogenous fuzzy closure spaces and lattices of fuzzy closure operators. In 2002, Sunil\cite{SUNIL} studied generalised-closed sets, and well closed fuzzy points in fuzzy closure spaces. He paid some attention to separation axioms such as $FT_1$ and quasi separation. But according to his definition, a fuzzy closure space is $FT_1$ (respectively, quasi separated) if its associated fuzzy topology is $FT_1$ (respectively, quasi separated). Further, he defined static fuzzy closure spaces, D-fuzzy closure spaces, and adjacent fuzzy closure operators. In 2007, Bloomy Joseph\cite{BLOOMY} studied denseness in fuzzy closure spaces and fuzzy closure fuzzy convexity spaces. In 2014, Tapi and Deole (\cite{TAPI}\cite{DEOLE}) introduced various forms of connectedness in fuzzy closure spaces and fuzzy biclosure spaces. Rachna Navalakhe \cite{TAPI2} introduced separation axioms such as Hausdorff, regular and normal in fuzzy closure spaces. They have used fuzzy open sets for the separation. So according to their definition, a fuzzy closure space satisfies a certain separation property if its associated fuzzy topology also posses the same property. Their definitions are not an extension of separation axioms in \v{C}ech closure spaces. Also, she has studied generalised-closed sets, generalised forms of separation axioms, $\delta $-fuzzy closed sets, separation axioms using $\delta$-fuzzy closed sets, and studied analogous concepts in fuzzy biclosure spaces.
			
			\par In this paper, we have introduced the interior operator and hence the neighborhood systems in fuzzy closure spaces. We have defined the notion of \v{C}F-continuity at a fuzzy point. While defining separation axioms on fuzzy closure spaces, we want to make it an extension of that of \v{C}ech closure spaces and fuzzy topological spaces. Many authors have introduced separation axioms in fuzzy topological spaces differently. In L-topology \cite{LIU}, it was introduced by using Q-neighborhoods. If we introduce separation axioms on fuzzy closure spaces using Q-neighborhoods, the closure operator will have no role. So we introduce it as an extension of separation properties of fuzzy topological spaces as introduced by Mashhur et al. \cite{MASH2}. Using the neighborhood systems, we have introduced separation properties such as \v{C}F$T_0$, \v{C}F$T_1$, \v{C}F$T_2,$ \v{C}F-Urysohn, \v{C}F-regular, and \v{C}F-normal. Additive and productive properties of these separation axioms have been studied. All these separation properties are fuzzy closure properties. The relationships between these separation axioms are also discussed. 
			
			\section{Preliminaries}

			This section contains some basic concepts and theorems about fuzzy topological spaces, \v{C}ech closure spaces, and fuzzy closure spaces which will be needed in the sequel. Throughout this paper $X, Y, Z,$ etc. denotes non-empty ordinary sets, $f,g,h,$ etc. denotes fuzzy sets defined on an ordinary set, and $x,y,x^1,x^2,x^3,$ etc. denotes elements of ordinary sets.

			\begin{Definition}{\cite{CECH}}
				A \v{C}ech closure operator (also called a \v{C}-closure operator, or simply a closure operator) on a set $X$ is a function $ c: \{0,1\}^{X} \rightarrow \{0,1\}^{X} $ satisfying the following axioms:
				\begin{itemize}
					\item [(i)]	{$ c(\emptyset) = \emptyset $},
					\item[(ii)] {$A \subset c(A) $ }, $\text{ for all } A \subseteq X$,
					\item[(iii)]{$  c(A \cup B) = c(A)\cup c(B) $}, $\text{ for all } A,B \subseteq X$.
				\end{itemize}
				The pair $(X,c)$ is called a \v{C}ech closure space or simply a closure space.    
			\end{Definition}

			If $c$ satisfies idempotency also (i.e., $c(c(A))=c(A)$, $\text{ for all } A \subseteq X$), it will become a topological closure operator. Separation properties, connectedness, and compactness in \v{C}ech closure spaces are studied in \cite{SUNI}.

			Let $X$ be an ordinary non-empty set. A fuzzy set $f$ on $X$ is a mapping from $X$ to the closed unit interval $I=[0,1]$ of the real line, associating each $x$ of $X$ with its membership value $f(x)$. The families $\{0,1\}^X$ and $I^X$ denote, respectively, the collections of all crisp subsets and all fuzzy subsets of $X$. The constant fuzzy sets $\underline{0}$ and $\underline{1}$ are given by assigning the value $0$ and $1$ to every element of $X$, respectively. The fuzzy set $\underline{0}$ represents the fuzzy empty set, while $\underline{1}$ corresponds to the fuzzy universal set. Let $A\subset X$, $1_A$ denotes the fuzzy subset of $X$ having membership value 1 on elements of $A$ and 0 otherwise. For any fuzzy set $f$ in $X$, by $\mathfrak{Co}(f)$ we mean the complement of $f$ in $X$ and is defined as, $\mathfrak{Co}(f)(x)= 1-f(x)$. Two fuzzy sets $f$ and $g$ of $X$ are equal if $f(x)=g(x)$ for all $x\in X$. If $f(x)\leq g(x) \hspace{0.2cm}\text{ for all } x \in X$, then $f$ is said to be a subset of $g$.

			Let $X$ be a non-empty set. A fuzzy set $f$ on $X$ is defined as a mapping $f:X \to I=[0,1]$, where $f(x)$ represents the degree of membership of each $x\in X$ in $f$. The families $\{0,1\}^X$ and $I^X$ denote, respectively, the collections of all crisp subsets and all fuzzy subsets of $X$. The constant fuzzy sets $\underline{0}$ and $\underline{1}$ are given by assigning the value $0$ and $1$ to every element of $X$, respectively. For any subset $A\subseteq X$, the characteristic fuzzy set $1_A$ is defined by $1_A(x)=1$ if $x\in A$ and $1_A(x)=0$ otherwise. For a fuzzy set $f$ on $X$, the complement $\mathfrak{Co}(f)$ is defined by $\mathfrak{Co}(f)(x)=1-f(x)$ for all $x\in X$. Two fuzzy sets $f$ and $g$ are said to be equal if $f(x)=g(x)$ for every $x\in X$. Moreover, if $f(x)\leq g(x)$ for all $x\in X$, then $f$ is regarded as a fuzzy subset of $g$.

			\begin{Definition}{\cite{CHAN}}
				A \emph{fuzzy topology} on a set $X$ is a family $F$ of fuzzy subsets of $X$ that satisfies the following conditions:
				\begin{itemize}
					\item[(i)] $\underline{0}, \underline{1} \in F$,
					\item[(ii)] if $g,h \in F$, then $g \wedge h \in F$,
					\item[(iii)] if $g_i \in F$ for each $i \in I$, then $\underset{i \in I}{\bigvee} g_i \in F$.
				\end{itemize}
				The pair $(X, F)$ is called a {fuzzy topological space} (abbreviated as fts). The elements of $F$ are referred to as {open fuzzy subsets} of $X$, while their complements are called {closed fuzzy subsets}. A fuzzy topological space $(X,F)$ with $F=\{\underline{0}, \underline{1}\}$ is said to be the {indiscrete} (or trivial) fuzzy topology, whereas if $F=[0,1]^X$, it is referred to as the {discrete} fuzzy topology.
			\end{Definition}

			\begin{Definition}\cite{CHAN}
				Let $f$ be a fuzzy subset of an fts $(X,F)$. The closure of $f$ is defined to be the fuzzy subset $\bigwedge \{g: f\leq g, \mathfrak{Co}(g)\in F\}$ and is denoted by $\overline{f}$.
			\end{Definition}
			
			\begin{Definition}\cite{CHAN}
				Let $X$ and $Y$ be two sets and let $\theta: X \rightarrow Y$ be a function. Then, for any fuzzy subset $g$ of $X$, $\theta(g)$ is a fuzzy subset in $Y$ defined by 
				
				\[
				\theta(g)(y) =
				\begin{cases}
					\sup \{\, g(x) : x \in X, \ \theta(x) = y \}, & \text{if } \theta^{-1}(y) \neq \emptyset, \\[6pt]
					0, & \text{if } \theta^{-1}(y) = \emptyset.
				\end{cases}
				\]
				
				For a fuzzy subset $h$ of $Y$, $\theta^{-1}(h)$ is a fuzzy subset of $X$ defined by $\theta^{-1}(h)(x)=h(\theta(x))$.
			\end{Definition}
			
			\begin{Definition}\cite{YING} 
				The fuzzy subset $x_\lambda$ of $X$, with $x \in X$ and $0 <\lambda \leq 1$ defined by, 
				
				\[
				x_\lambda(y) =
				\begin{cases}
					\lambda, & \text{if } y = x, \\[6pt]
					0, & \text{otherwise}.
				\end{cases}
				\]
				
				is called a fuzzy point in $X$ with support $x$ and value $\lambda$. Two fuzzy points with different supports are called distinct. When $\lambda=1$, $x_\lambda = x_1$ is called a fuzzy singleton.
			\end{Definition}

			\begin{Definition}\cite{SUNIL2}
				Let $(X, F)$ be a fuzzy topological space and $x_\lambda$ be a fuzzy point in $X$. If $\overline{{x_\lambda}}$ is again a fuzzy point, it is said to be well closed.
			\end{Definition}
			
			\begin{Definition}\cite{LIU}
				Let $f$ be a fuzzy subset of $X$, it's support is defined as $\operatorname{supp}(f)= \{x: f(x)> 0\}$.
			\end{Definition}
			
			\begin{Definition}\cite{MASH2}
				A fuzzy topological space $(X,F)$ is said to be $FT_0$ if for every pair of distinct fuzzy points $x_\lambda$ and $y_\gamma$, there exist an $f\in F$ such that $x_\lambda \in f \leq \mathfrak{Co}(y_\gamma)$ or  $y_\gamma \in f \leq \mathfrak{Co}(x_\lambda)$. 
			\end{Definition}

			\begin{Definition}\cite{MASH2}
				A fuzzy topological space $(X,F)$ is said to be $FT_1$ if and only if for every pair of distinct fuzzy points $x_\lambda$ and $y_\gamma$, there exist open sets $f,g\in F$ such that $x_\lambda \in f \leq \mathfrak{Co}(y_\gamma)$ and  $y_\gamma \in g \leq \mathfrak{Co}(x_\lambda)$.
			\end{Definition}
			
			\begin{Theorem}\cite{KERR}\label{thm:FT1}
				A fuzzy topological space $(X, F)$ is $FT_1$ if and only if every fuzzy singleton is closed.
			\end{Theorem}
			
			\begin{Definition}\cite{MASH2}
				A fuzzy topological space $(X, F)$ is said to be $FT_s$ (strong $FT_1$) if every fuzzy point is closed.
			\end{Definition}


			\begin{Definition}\cite{MASH2}
				A fuzzy topological space $(X,F)$ is $FT_2$ if for every pair of distinct fuzzy points $x_\lambda$ and $y_\gamma$, there exists $f,g\in F$ such that $f \leq \mathfrak{Co}(g)$ and $x_\lambda \in f \leq \mathfrak{Co}(y_\gamma)$ and  $y_\gamma \in g \leq \mathfrak{Co}(x_\lambda)$.
			\end{Definition}

			\begin{Definition}\cite{MASH2}
				A fuzzy topological space $(X, F)$ with the closure operator $c$ is said to be $FT_{2\frac{1}{2}}$ or fuzzy Urysohn if for every pair of distinct fuzzy points $x_\lambda$ and $y_\gamma$, there exist $f,g\in F$ such that $c(f) \leq \mathfrak{Co}(c(g))$ and $x_\lambda \in f \leq \mathfrak{Co}(y_\gamma)$ and  $y_\gamma \in g \leq \mathfrak{Co}(x_\lambda)$. 
			\end{Definition}

			\begin{Definition}\cite{MASH2}
				A fuzzy topological space $(X, F)$ is said to be fuzzy regular if for all fuzzy point $x_\lambda$, and every closed fuzzy set $k$ in $X$ such that $x_\lambda \in \mathfrak{Co}(k)$,  there exist $f,g\in F$ such that $x_\lambda \in f $ and  $k \leq g $ and $f \leq \mathfrak{Co}(g)$. A fuzzy regular space which is also $FT_s$ is said to be $FT_3$.
			\end{Definition}

			\begin{Definition}\cite{MASH2}
				A fuzzy topological space $(X,F)$ is said to be fuzzy normal if for all closed fuzzy sets $k_1,k_2$ in $X$ such that $k_1 \leq \mathfrak{Co}(k_2)$, there exists $f_1,f_2\in F$ such that $k_1 \leq  f_1 $ and  $k_2 \leq f_2 $ and $f_1 \leq \mathfrak{Co}(f_2)$. A fuzzy normal space which is also $FT_s$ is said to be $FT_4$.
			\end{Definition}

			\begin{Definition}\cite{MASH}
				A \v{C}ech fuzzy closure operator (abbreviated as \v{C}F-closure operator, or simply a fuzzy closure operator) on a set $X$ is a mapping $c: I^{X} \rightarrow I^{X}$ that satisfies the following axioms:
				\begin{itemize}
					\item[(i)] $c(\underline{0}) = \underline{0}$,
					\item[(ii)] $f \leq c(f)$ for all $f \in I^X$,
					\item[(iii)] $c(f \vee g) = c(f) \vee c(g)$ for all $f,g \in I^X$.
				\end{itemize}
				The pair $(X,c)$ is called a \v{C}ech fuzzy closure space (or fuzzy closure space, abbreviated as fcs). If, in addition, $c$ satisfies the idempotent property, i.e., $c(c(f))=c(f)$ for all fuzzy subsets $f$ of $X$, then $c$ is said to be fuzzy topological.
			\end{Definition}

			\begin{example}
				The fuzzy closure operator $i$ defined by \hspace{0.1cm}$
				i(\underline{0}) = \underline{0}, \hspace{0.1cm}  i(f) = \underline{1} $  for every non-empty fuzzy subset $f$ of $X,$
				is called the {indiscrete fuzzy closure operator}.  
				
				Similarly, the fuzzy closure operator $d$ defined by $d(f) = f  \text{  for all fuzzy subsets } f \text{ of } X,$ is called the {discrete fuzzy closure operator}.
			\end{example}
			
			A \v{C}ech closure space $(X,c)$ is said to be finitely generated\cite{ROTH} if $c(A)=\underset{x\in A}{\bigcup} c(x)$ for every. We can define an analogous concept in fuzzy closure spaces as follows.    	
			\begin{Definition}
				An fcs $(X,c)$ is said to be finitely generated if $c(f)= \underset{x_\lambda \leq f}{\bigvee}c(x_\lambda)$ for all $f\in I^X$.
			\end{Definition}
			The essential thing is that a finitely generated fuzzy closure operator is completely determined by its action on fuzzy points.
			
			\begin{Definition}\cite{MASH}
				Let  $(X,c)$ be a fuzzy closure space (fcs), the fts $(X,\tau(c))$, where $\tau(c)= \{f: c(\mathfrak{Co}(f))=\mathfrak{Co}(f)\}$ is called the fts associated with  $(X,c)$. A fuzzy subset $f$ of $X$ is said to be closed in $(X,c)$ if $c(f)=f$. 
			\end{Definition}
			
			Different fuzzy closure operators can associate the same fuzzy topology. For example, consider $\mathbb{Z}$ with a finitely generated fuzzy closure operator defined as $c(x_\lambda)= x_\lambda \vee (x+1)_\lambda$ for all $x\in \mathbb{Z}$, and the indiscrete fuzzy closure operator. Both fuzzy closure operators have the indiscrete fuzzy topology as their associated fuzzy topology.
			\begin{Definition}\cite{MASH}
				Let $(X,c)$ be an fcs and $A$ be an ordinary subset of $X$. The function $c_A: I^X \rightarrow I^X$ defined by $c_A(f)= 1_A \wedge c(f)$ is a \v{C}F-closure operator on $A$. The corresponding pair $(A,c_A)$ is said to be a \v{C}F-closure subspace of $(X,c)$. 
			\end{Definition}
			
			\begin{Definition}\cite{MASH}
				Let $c_1\text{ and }c_2$ be two fuzzy closure operators on a set $X$. Then, the fcs $(X,c_1)$ is said to be coarser than the fcs $(X,c_2)$ if $c_2(f) \leq c_1(f)$ for all $f \in I^X$ and denoted by  $c_1 \leq c_2$.
			\end{Definition}

			\begin{Definition}{\cite{MASH}}
				Let $(X,c)$ and $(Y,d)$ be fuzzy closure spaces. A mapping $\theta : X \to Y$ is called \v{C}F-continuous if 
				$\theta(c(f)) \leq d(\theta(f)),  \text{  for all } f \in I^X.$ If $\theta$ is bijective and both $\theta$ and $\theta^{-1}$ are \v{C}F-continuous, then $\theta$ is termed a \v{C}F-homeomorphism.  
				
				A fuzzy closure property refers to a structural property of fuzzy closure spaces that remains preserved under \v{C}F-homeomorphisms.
			\end{Definition}

			\begin{Definition}{\cite{MASH}}
				A fuzzy closure property is said to be hereditary; whenever a fuzzy closure space has that property, then so does every subspace of it.
			\end{Definition}
			\begin{Definition}\cite{MASH}
				Let $\mathcal{F}= \{(X_t,c_t): t \in  T\}$ be a family of pairwise disjoint fuzzy closure spaces and $X= \underset{t \in T}{\bigvee} X_t$. The function $\oplus : I^X \rightarrow I^X$, defined by, $\oplus c_t(f)=  \underset{t \in T}{\bigvee} c_t(1_{X_t} \wedge f)$ is a \v{C}F-closure operator on $X$. The pair $(X,\oplus c_t)$ is said to be the sum fuzzy closure space of the family $\mathcal{F}$.
			\end{Definition}
			
			\begin{Definition}\cite{MASH}\label{thm:prod}
				Let $\mathcal{F}= \{(X_t,c_t): t \in  T\}$ be a family of fuzzy closure spaces and $X= \underset{t \in T}{\prod} X_t$. Define a function $\otimes c_t: I^X \rightarrow I^X$ as, for a fuzzy point $x_\lambda$, and a fuzzy set $f$ in $X$, $x_\lambda \leq \otimes c_t(f)$ if the following condition is satisfied. $f= f_1 \vee f_2 \vee \cdots \vee f_n, (f_i \in I^X)$ implies there exists an $i$ such that $ P_t(x_\lambda) \leq c_t(P_t(f_i))$ for all $t$. The function $ \otimes c_t $ is a \v{C}F-closure operator on $X$. The pair $(X,\otimes c_t)$ is said to be the product fuzzy closure space of the family $\mathcal{F}$. Where $P_t$ is the projection from $X$ to $X_t$.
			\end{Definition}
			
			\begin{Theorem}\cite{MASH}
				The product fcs is the coarsest fcs for which each projection is \v{C}F-continuous. 
			\end{Theorem}
			
			\section{Interior operator and Neighborhood systems}
			
			In this section, we introduce interior operators and neighborhood systems in fuzzy closure spaces and give some equivalent conditions for \v{C}F-continuity and \v{C}F-homeomorphism.
			\begin{Definition}
				An interior operator ` $\operatorname{int}$' is a function from $I^X$ to itself defined by, $\intt(f)= \mathfrak{Co}(c(\mathfrak{Co}(f)))$, for each fuzzy subset $f$ of $X$. A fuzzy subset $f$ of $X$ is said to be a neighborhood of $x_\lambda$ if $x_\lambda  \leq \intt(f)$, if so we say that $x_\lambda$ is an interior point of $f$. Furthermore, for a fuzzy subset $k $ of $ X$, $f$ is called a neighborhood of $k$ if $k \leq \operatorname{int}(f)$.
			\end{Definition}

			Mashhur and Ghanim\cite{MASH} also used $f^*=\mathfrak{Co}(c(\mathfrak{Co}(f)))$, which is same as our interior operator for studying regularity and \v{C}F-strongly continuous maps. But they haven't identified it as an interior operator. 
			\begin{example}
				We can easily find that the interior operator of indiscrete fcs is  $\intt_i(f)=\underline{0} \text{  if  } f\neq \underline{1}$, and $\intt_i(\underline{1})= \underline{1}$. 
				Similarly, the interior operator of discrete fcs is, $\intt_d(f)=f $ for all fuzzy sets $f$ in $X$.
			\end{example}

			\begin{example}
				Let $X=\{p,q,r\}$ and define a closure operator $c:I^X \to I^X$ on $X$ by
				\[
				c(f)=
				\begin{cases}
					\underline{0}, & \text{if } f=\underline{0}, \\[6pt]
					1_{\{p,q\}}, & \text{if } \underline{0} < f \leq 1_{\{p\}}, \\[6pt]
					1_{\{p,q,r\}}, & \text{otherwise}.
				\end{cases}
				\]
				The associated interior operator $\operatorname{int}:I^X \to I^X$ can be found as
				\[
				\operatorname{int}(f)=
				\begin{cases}
					\underline{1}, & \text{if } f=\underline{1}, \\[6pt]
					1_{\{r\}}, & \text{if } \underline{1} > f \geq 1_{\{q,r\}}, \\[6pt]
					\underline{0}, & \text{otherwise}.
				\end{cases}
				\]
			\end{example}

			\begin{example}
				On $\mathbb{Z}$, define a fuzzy closure operator $c:I^{\mathbb{Z}} \to I^{\mathbb{Z}}$ as follows.  
				For each fuzzy point $x_\lambda \in I^{\mathbb{Z}}$, set
				\[
				c(x_\lambda) = 1_{\{x,\,x+1\}},
				\]
				and for an arbitrary fuzzy subset $f \in I^{\mathbb{Z}}$, define
				\[
				c(f) = \bigvee_{x_\lambda \leq f} c(x_\lambda).
				\]
				Here we get, $
				\operatorname{int}(1_{\{x,\,x+1\}}) = 1_{\{x\}}$, and it follows that every neighborhood of the fuzzy point $x_\lambda$ necessarily contains $1_{\{x,\,x+1\}}$.
			\end{example}

			\begin{remark}
				If the interior operator is given, we can find the closure operator as follows: $c(f)= \mathfrak{Co}(\intt(\mathfrak{Co}(f)))$.
			\end{remark}
			Analogous to \v{C}ech closure spaces, the following results hold.

			\begin{Theorem}
				Let $(X,c)$ be a fuzzy closure space. Then, for all $f,g \in I^X$, the following properties of the fuzzy interior operator hold:
				\begin{enumerate}
					\item $\operatorname{int}(\underline{1}) = \underline{1}$,
					\item $\operatorname{int}(f) \leq f$,
					\item $\operatorname{int}(f \wedge g) = \operatorname{int}(f) \wedge \operatorname{int}(g)$, \label{eq:int3}
					\item If $f \leq g$, then $\operatorname{int}(f) \leq \operatorname{int}(g)$,
					\item A fuzzy subset $f$ of $X$ is open if and only if $\operatorname{int}(f) = f$,
					\item A fuzzy subset $f$ of $X$ is open if and only if $f$ is a neighborhood of every fuzzy point $x_\lambda \in f$.
				\end{enumerate}
			\end{Theorem}
			
			\begin{proof}
				\begin{enumerate}
					\item By definition,
					\[
					\operatorname{int}(\underline{1}) = \mathfrak{Co}(c(\mathfrak{Co}(\underline{1}))) 
					= \mathfrak{Co}(c(\underline{0})) 
					= \mathfrak{Co}(\underline{0}) 
					= \underline{1}.
					\]
					
					\item Since $\mathfrak{Co}(f) \leq c(\mathfrak{Co}(f))$, it follows that
					\[
					\operatorname{int}(f) = \mathfrak{Co}(c(\mathfrak{Co}(f))) \leq f.
					\]
					
					\item Using the properties of the complement operator, we have
					\[
					\mathfrak{Co}(f \vee g) = \mathfrak{Co}(f) \wedge \mathfrak{Co}(g), \quad 
					\mathfrak{Co}(f \wedge g) = \mathfrak{Co}(f) \vee \mathfrak{Co}(g).
					\]
					Hence,
					\[
					\begin{aligned}
						\operatorname{int}(f \wedge g) 
						&= \mathfrak{Co}(c(\mathfrak{Co}(f \wedge g))) \\
						&= \mathfrak{Co}(c(\mathfrak{Co}(f) \vee \mathfrak{Co}(g))) \\
						&= \mathfrak{Co}(c(\mathfrak{Co}(f)) \vee c(\mathfrak{Co}(g))) \\
						&= \mathfrak{Co}(c(\mathfrak{Co}(f))) \wedge \mathfrak{Co}(c(\mathfrak{Co}(g))) \\
						&= \operatorname{int}(f) \wedge \operatorname{int}(g).
					\end{aligned}
					\]
					
					\item If $f \leq g$, then $f \wedge g = f$. By property \ref{eq:int3}, 
					\[
					\operatorname{int}(f) = \operatorname{int}(f \wedge g) 
					= \operatorname{int}(f) \wedge \operatorname{int}(g) \leq \operatorname{int}(g).
					\]
					
					\item Suppose $f$ is open in $(X,c)$. Then $\mathfrak{Co}(f)$ is closed, i.e., $c(\mathfrak{Co}(f)) = \mathfrak{Co}(f)$. Thus,
					\[
					\operatorname{int}(f) = \mathfrak{Co}(c(\mathfrak{Co}(f))) = f.
					\]
					Conversely, if $\operatorname{int}(f) = f$, then
					\[
					\mathfrak{Co}(c(\mathfrak{Co}(f))) = f.
					\]
					Taking complements, we obtain $c(\mathfrak{Co}(f)) = \mathfrak{Co}(f)$. Hence, $\mathfrak{Co}(f)$ is closed, and $f$ is open.
					
					\item If $f$ is open, then $\operatorname{int}(f) = f$, so $f$ is a neighborhood of each fuzzy point contained in it. Conversely, suppose $f$ is a neighborhood of every fuzzy point $x_\lambda \in f$. Then, we have
					\[
					x_\lambda \leq \operatorname{int}(f) \leq f \quad \text{for all } x_\lambda \in f, 
					\]
					which implies,
					\[
					f=\bigvee_{x_\lambda \in f} x_\lambda \leq \operatorname{int}(f) \leq f. 
					\]
					Therefore, $\operatorname{int}(f) = f$, and hence $f$ is open. \qedhere
				\end{enumerate}
			\end{proof}

			
			Here we introduce the notion of \v{C}F-continuity at a fuzzy point and subsequently employ it to derive an equivalent criterion for \v{C}F-continuity.
			
			\begin{Definition}
				A function $\theta:(X,c) \rightarrow (Y,d)$ is said to be \v{C}F-continuous at a fuzzy point $x_\lambda$ if $ x_\lambda \in c(f)$ for some fuzzy subset $f$ of $X$, then $ \theta(x_\lambda)\leq d(\theta(f))$.
			\end{Definition}

			\begin{Theorem}
				A function $\theta:(X,c) \to (Y,d)$ is \v{C}F-continuous if and only if it is \v{C}F-continuous at every fuzzy point of $X$.
			\end{Theorem}
			
			\begin{proof}
				($\Rightarrow$) Suppose $\theta$ is \v{C}F-continuous. Then, by definition, 
				\[
				\theta(c(f)) \leq d(\theta(f)), \quad \forall f \in I^X.
				\]
				Let $x_\lambda$ be a fuzzy point in $X$. If $x_\lambda \in c(f)$, then 
				\[
				\theta(x_\lambda) \in \theta(c(f)) \leq d(\theta(f)).
				\]
				Thus, $\theta$ is \v{C}F-continuous at the fuzzy point $x_\lambda$.
				
				($\Leftarrow$) Conversely, assume that $\theta$ is \v{C}F-continuous at every fuzzy point of $X$. For any fuzzy subset $f \in I^X$, we have
				\[
				\theta(c(f)) = \bigvee_{x_\lambda \in c(f)} \theta(x_\lambda).
				\]
				Since $\theta$ is \v{C}F-continuous at each fuzzy point, it follows that $\theta(x_\lambda) \in d(\theta(f))$ whenever $x_\lambda \in c(f)$. Consequently,
				\[
				\theta(c(f)) = \bigvee_{x_\lambda \in c(f)} \theta(x_\lambda)\quad  \leq d(\theta(f)).
				\]
				Hence, $\theta$ is \v{C}F-continuous.
			\end{proof}

			Another characterization theorem for \v{C}F-continuous functions can be found as follows:

			\begin{Theorem}
				Let $(X,c)$ and $(Y,d)$ be fuzzy closure spaces. A mapping $\theta:(X,c) \to (Y,d)$ is \v{C}F-continuous if and only if 
				\[
				c(\theta^{-1}(g)) \leq \theta^{-1}(d(g)), \quad \text{for all } g \in I^Y.
				\]
			\end{Theorem}
			
			\begin{proof}
				($\Rightarrow$) Suppose that $\theta:(X,c) \to (Y,d)$ is \v{C}F-continuous. By definition, we have
				\[
				\theta(c(f)) \leq d(\theta(f)), \quad \text{for all } f \in I^X.
				\]
				Taking $f = \theta^{-1}(g)$ for an arbitrary $g \in I^Y$, it follows that
				\[
				\theta(c(\theta^{-1}(g))) \leq d(\theta(\theta^{-1}(g))) \leq d(g).
				\]
				Applying $\theta^{-1}$ to both sides yields
				\[
				c(\theta^{-1}(g)) \leq \theta^{-1}(d(g)), \quad \forall g \in I^Y.
				\]
				
				($\Leftarrow$) Conversely, assume that
				\[
				c(\theta^{-1}(g)) \leq \theta^{-1}(d(g)), \quad \forall g \in I^Y.
				\]
				Let $f \in I^X$, and set $g = \theta(f)$ and $f_1 = \theta^{-1}(g)$. Then
				\[
				c(f_1) = c(\theta^{-1}(g)) \leq \theta^{-1}(d(g)).
				\]
				Since $f \leq f_1$, we obtain
				\[
				\theta(c(f)) \leq \theta(c(f_1)) \leq \theta(\theta^{-1}(d(g))) = d(g) = d(\theta(f)).
				\]
				Therefore, $\theta$ is \v{C}F-continuous.
			\end{proof}
			
			\begin{corollary}
				If $\theta:(X,c) \to (Y,d)$ is \v{C}F-continuous, then the inverse image of each open (respectively, closed) fuzzy subset of $Y$ is open (respectively, closed) in $X$.
			\end{corollary}
			
			\begin{proof}
				Let $g$ be a closed fuzzy subset of $(Y,d)$, i.e., $d(g) = g$. By the above theorem, we have
				\[
				c(\theta^{-1}(g)) \leq \theta^{-1}(d(g)) = \theta^{-1}(g).
				\]
				Thus, $c(\theta^{-1}(g)) = \theta^{-1}(g)$, showing that $\theta^{-1}(g)$ is closed in $(X,c)$. By duality, the result also holds for open fuzzy subsets.
			\end{proof}

			\begin{remark}
				The converse of the preceding corollary does not hold in general. Consider the set 
				$X = \{p,q,r,s\}$ equipped with the fuzzy closure operator $c$ defined by
				\[
				c(f) =
				\begin{cases}
					\underline{0}, & f = \underline{0}, \\[6pt]
					1_{\{p,q\}}, & \underline{0} < f \leq 1_{\{p\}}, \\[6pt]
					1_{\{q,r\}}, & \underline{0} < f \leq 1_{\{q\}}, \\[6pt]
					1_{\{r,s\}}, & \underline{0} < f \leq 1_{\{r\}}, \\[6pt]
					1_{\{s,p\}}, & \underline{0} < f \leq 1_{\{s\}}, \\[6pt]
					\bigvee_{x_\lambda \in f} c(x_\lambda), & \text{otherwise}.
				\end{cases}
				\]
				
				Define a mapping $\theta : (X,c) \to (X,c)$ by
				\[
				\theta(p) = q, \quad \theta(q) = r, \quad \theta(r) = s, \quad \theta(s) = p.
				\]
				Since the fuzzy topology associated with $(X,c)$ is the indiscrete fuzzy topology, the inverse image of every open set under $\theta$ is again open.  
				
				However, consider $f = 1_{\{q\}}$. Then
				\[
				\theta(c(1_{\{q\}})) = \theta(1_{\{q,r\}}) = 1_{\{p,r,s\}},
				\]
				while
				\[
				c(\theta(1_{\{q\}})) = c(1_{\{r\}}) = 1_{\{r,s\}}.
				\]
				Thus,
				\[
				\theta(c(1_{\{q\}})) \nleq c(\theta(1_{\{q\}})),
				\]
				showing that $\theta$ fails to be \v{C}F-continuous. 
			\end{remark}

			\begin{Theorem}
				Let $(X,c)$ and $(Y,d)$ be fuzzy closure spaces. A bijective function $\theta : (X,c) \to (Y,d)$ is a \v{C}F-homeomorphism if and only if $\theta(c(f)) = d(\theta(f)),  \text{  for all } f \in I^X.$
			\end{Theorem}
			
			\begin{proof}
				($\Rightarrow$) Assume that $\theta : (X,c) \to (Y,d)$ is a \v{C}F-homeomorphism. Then, by definition,  
				\[
				\theta(c(f)) \leq d(\theta(f)) \quad \text{and} \quad \theta^{-1}(d(g)) \leq c(\theta^{-1}(g)),
				\]
				for all $f \in I^X$ and $g \in I^Y$.  
				Substituting $g = \theta(f)$, we obtain
				\[
				\theta^{-1}(d(\theta(f))) \leq c(\theta^{-1}(\theta(f))) = c(f),
				\]
				which implies
				\[
				d(\theta(f)) \leq \theta(c(f)), \quad \text{for all } f \in I^X.
				\]
				Combining this with $\theta(c(f)) \leq d(\theta(f))$, we conclude that
				\[
				\theta(c(f)) = d(\theta(f)), \quad \text{for all } f \in I^X.
				\]
				
				($\Leftarrow$) Conversely, suppose that $\theta(c(f)) = d(\theta(f))$ for all $f \in I^X$. Clearly, $\theta(c(f)) \leq d(\theta(f)) \text{  for all } f \in I^X$, hence $\theta$ is \v{C}F-continuous.  
				Also we have $d(\theta(f)) \leq \theta(c(f))$, replacing $f$ by $\theta^{-1}(g)$, we get
				\[
				d(\theta(\theta^{-1}(g))) \leq \theta(c(\theta^{-1}(g))),
				\]
				which simplifies to
				\[
				d(g) \leq \theta(c(\theta^{-1}(g))), \quad \text{for all } g \in I^Y.
				\]
				Therefore,
				\[
				\theta^{-1}(d(g)) \leq c(\theta^{-1}(g)), \quad \text{for all } g \in I^Y,
				\]
				showing that $\theta^{-1}$ is also \v{C}F-continuous. Thus, $\theta$ is a \v{C}F-homeomorphism.
			\end{proof}

			\section{\textbf{\v{C}F$T_0$} fuzzy closure spaces}
			
			\vspace{6pt}
			\begin{Definition}
				A fuzzy closure space $(X,c)$ is \v{C}F$T_0$ if for every two distinct fuzzy points $x_\lambda$ and $y_{\gamma}$, either $x_\lambda \in \mathfrak{Co}( c(y_\gamma))$ or $ y_\gamma \in \mathfrak{Co} (c(x_\lambda))$.
			\end{Definition}
			It can be noted that if $c$ is a fuzzy topological closure operator, then the fuzzy closure space $(X,c)$ satisfies the \v{C}F$T_0$ axiom if and only if the associated fuzzy topological space $(X, \tau(c))$ is $FT_0$.
			
			\begin{remark}
				Every subspace of a \v{C}F$T_0$ space is \v{C}F$T_0$. (That is, \v{C}F$T_0$ is a hereditary property.)
			\end{remark}
			
			\begin{remark}
				Let $c_1$ and $c_2$ be two fuzzy closure operators on $X$ such that $c_1 \leq c_2$. If $(X,c_1)$ is \v{C}F$T_0$, then $(X,c_2)$ is also \v{C}F$T_0$.
			\end{remark}

			\begin{Theorem}
				If $(X, \tau(c))$ is $FT_0$, then $(X,c)$ is \v{C}F$T_0$.
			\end{Theorem}
			
			\begin{proof}
				Assume that $(X, \tau(c))$ is $FT_0$, and let $x_\lambda, y_\gamma \in I^X$ be two fuzzy points.  
				Since $(X, \tau(c))$ is $FT_0$, there exists an $f \in \tau(c)$ such that 
				\[
				x_\lambda \in f \quad \text{and} \quad f \leq \mathfrak{Co}(y_\gamma).
				\]
				This yields
				\[
				\mathfrak{Co}(x_\lambda) \geq \mathfrak{Co}(f) \geq y_\gamma.
				\]
				
				Moreover, as $f$ is open, $\mathfrak{Co}(f)$ is closed. Thus, if $y_\gamma \leq \mathfrak{Co}(f)$, then
				\[
				y_\gamma \leq c(y_\gamma) \leq \mathfrak{Co}(f).
				\]
				Consequently,
				\[
				\mathfrak{Co}(x_\lambda) \geq \mathfrak{Co}(f) \geq c(y_\gamma) \geq y_\gamma.
				\]
				From this, we can deduce that $x_\lambda \in \Co(c(y_\gamma ))$.
				Hence $(X,c)$ satisfies the \v{C}F$T_0$ property.
			\end{proof}

			\begin{remark}
				The converse of the preceding theorem does not hold in general.  
				Indeed, consider the following example: 
				
				Let $X=\{x,y,z\}$ and define a finitely generated fuzzy closure operator $c$ on $X$ by 
				\[
				c(\underline{0})=\underline{0}, \quad 
				c(x_\lambda)=1_{\{x,y\}}, \quad 
				c(y_\gamma)=1_{\{y,z\}}, \quad 
				c(z_\rho)=1_{\{z,x\}} \quad \text{for all } 0<\lambda,\gamma,\rho\leq 1,
				\]
				and for any other fuzzy subset $f$ of $X$, define
				\[
				c(f)=\bigvee_{u_\lambda \in f} c(u_\lambda).
				\]
				
				It is straightforward to verify that $c$ is a fuzzy closure operator, although it is not induced by any fuzzy topology.  
				Indeed, we observe that
				\[
				x_\lambda \in \mathfrak{Co}(c(y_\gamma))=x_1, \quad 
				z_\rho \in \mathfrak{Co}(c(x_\lambda))=z_1, \quad 
				y_\gamma \in \mathfrak{Co}(c(z_\rho))=y_1.
				\]
				Hence, the fuzzy closure space $(X,c)$ satisfies the \v{C}F$T_0$ separation axiom. However,  $	\tau(c)=\{\underline{0}, \underline{1}\},$
				the indiscrete fuzzy topology, which is not $FT_0$.  
			\end{remark}

			\begin{Theorem}
				The property \v{C}F$T_0$ is preserved under \v{C}F-Homeomorphisms; that is, \v{C}F$T_0$ is a fuzzy closure property.
			\end{Theorem}
			
			\begin{proof}
				Let $\theta:(X,c) \to (Y,d)$ be a \v{C}F-homeomorphism, and assume that $(X,c)$ is \v{C}F$T_0$.  
				Consider two distinct fuzzy points $y^1_\lambda$ and $y^2_\gamma$ in $(Y,d)$. Then there exist $x^1, x^2 \in X$ such that $\theta(x^1) = y^1$ and $\theta(x^2) = y^2$.  
				
				Since $(X,c)$ is \v{C}F$T_0$, we may assume that  $x^1_\lambda \in \mathfrak{Co}_X(c(x^2_\gamma)).	$
				Applying $\theta$, we obtain
				\[
				y^1_\lambda \in \theta\big(\mathfrak{Co}_X(c(x^2_\gamma))\big)
				= \mathfrak{Co}_Y(\theta(c(x^2_\gamma))) 
				= \mathfrak{Co}_Y(d(\theta(x^2_\gamma))) 
				= \mathfrak{Co}_Y(d(y^2_\gamma)).
				\]
				Thus, $(Y,d)$ also satisfies the \v{C}F$T_0$ condition. Therefore, \v{C}F$T_0$ is a fuzzy closure property.
			\end{proof}

			\begin{Theorem}
				A fuzzy closure space $(X,c)$ is \v{C}F$T_0$ if and only if for every pair of distinct fuzzy points $x_\lambda$ and $y_\gamma$, there exists an $f \in I^X$ such that either 
				\[
				x_\lambda \in \operatorname{int}(f) \ \text{ and } \ y_\gamma \in \mathfrak{Co}(f), 
				\quad \text{or} \quad 
				y_\gamma \in \operatorname{int}(f) \ \text{ and } \ x_\lambda \in \mathfrak{Co}(f).
				\]
			\end{Theorem}
			
			\begin{proof}
				($\Rightarrow$) Suppose $(X,c)$ is \v{C}F$T_0$, and let $x_\lambda$ and $y_\gamma$ be two distinct fuzzy points in $X$. By the \v{C}F$T_0$ property, we may assume 
				\[
				x_\lambda \in \mathfrak{Co}(c(y_\gamma)) = \mathfrak{Co}\big(c(\mathfrak{Co}(\mathfrak{Co}(y_\gamma)))\big).
				\]
				Let $f = \mathfrak{Co}(y_\gamma)$. Then $y_\gamma \in \mathfrak{Co}(f)$, and since $x_\lambda \in \mathfrak{Co}(c(\mathfrak{Co}(f))) = \operatorname{int}(f)$, we obtain the desired condition.
				
				\medskip
				
				($\Leftarrow$) Conversely, assume that for distinct fuzzy points $x_\lambda$ and $y_\gamma$ there exists $f \in I^X$ such that $x_\lambda \in \operatorname{int}(f)$ and $y_\gamma \in \mathfrak{Co}(f)$. Since $y_\gamma \leq \mathfrak{Co}(f)$, we have
				$c(y_\gamma) \leq c(\mathfrak{Co}(f))$ which implies $ \mathfrak{Co}(c(y_\gamma)) \geq \mathfrak{Co}(c(\mathfrak{Co}(f))).$
				But $\operatorname{int}(f) = \mathfrak{Co}(c(\mathfrak{Co}(f)))$, and hence
				\[
				x_\lambda \in \operatorname{int}(f) \leq \mathfrak{Co}(c(y_\gamma)).
				\]
				Therefore, the \v{C}F$T_0$ condition holds.
			\end{proof}

			The following theorems show the additive and productive properties of the \v{C}F$T_0$ separation axiom.

			\begin{Theorem}
				Let $\mathcal{F}= \{(X_t,c_t): t \in T\}$ be a family of pairwise disjoint fuzzy closure spaces. Let $(X,c) =\big( \underset{t \in T}{\bigvee} X_t, \oplus c_t\big)$ denote their sum. Then $(X,c)$ is \v{C}F$T_0$ if and only if each $(X_t,c_t)$ is \v{C}F$T_0$.
			\end{Theorem}
			
			\begin{proof}
				($\Rightarrow$) Assume that each $(X_t,c_t)$ is \v{C}F$T_0$. Let $x_\lambda$ and $y_\gamma$ be two distinct fuzzy points in $I^X$. We consider two cases:
				
				\textit{Case (i):} Suppose $x,y \in X_{t_1}$ for some $t_1 \in T$. Since $(X_{t_1},c_{t_1})$ is \v{C}F$T_0$, either 
				\[
				x_\lambda \in \mathfrak{Co}_{X_{t_1}}(c_{t_1}(y_\gamma)) \quad \text{or} \quad y_\gamma \in \mathfrak{Co}_{X_{t_1}}(c_{t_1}(x_\lambda)).
				\]
				Now, note that
				\[
				c(x_\lambda) = \oplus c_t(x_\lambda) = \bigvee_{t \in T} c_t(x_\lambda \wedge 1_{X_t}) = c_{t_1}(x_\lambda),
				\]
				and similarly $c(y_\gamma) = c_{t_1}(y_\gamma)$. Hence, it follows that either 
				\[
				x_\lambda \in \mathfrak{Co}_X(c(y_\gamma)) \quad \text{or} \quad y_\gamma \in \mathfrak{Co}_X(c(x_\lambda)).
				\]
				
				\textit{Case (ii):} Suppose $x \in X_{t_1}$ and $y \in X_{t_2}$ with $t_1 \neq t_2$. Then, by construction,
				\[
				x_\lambda \in \mathfrak{Co}(1_{X_{t_2}}) \leq \mathfrak{Co}(c(y_\gamma)) \quad \text{and} \quad y_\gamma \in \mathfrak{Co}(1_{X_{t_1}}) \leq \mathfrak{Co}(c(x_\lambda)).
				\]
				Thus, $(X,c)$ is \v{C}F$T_0$.
				
				\medskip
				
				($\Leftarrow$) Conversely, suppose that $(X,c)$ is \v{C}F$T_0$. Let $x_\lambda, y_\gamma$ be two distinct fuzzy points in $I^{X_{t_1}}$ for some $t_1 \in T$. Considering $x_\lambda$ and $y_\gamma$ as fuzzy points of $(X,c)$, the \v{C}F$T_0$ property ensures that
				\[
				x_\lambda \in \mathfrak{Co}(c(y_\gamma)) \quad \text{or} \quad y_\gamma \in \mathfrak{Co}(c(x_\lambda)).
				\]
				Since
				\[
				c(x_\lambda) = \oplus c_t(x_\lambda) = c_{t_1}(x_\lambda), \quad c(y_\gamma) = \oplus c_t(y_\gamma) = c_{t_1}(y_\gamma),
				\]
				it follows that either 
				\[
				x_\lambda \in \mathfrak{Co}_{X_{t_1}}(c_{t_1}(y_\gamma)) \quad \text{or} \quad y_\gamma \in \mathfrak{Co}_{X_{t_1}}(c_{t_1}(x_\lambda)).
				\]
				Therefore, $(X_{t_1},c_{t_1})$ is \v{C}F$T_0$ for all $t_1 \in T$.
			\end{proof}

			\begin{lemma}\label{thm:prod2}
				Let $\mathcal{F} = \{(X_t,c_t): t \in T\}$ be a family of fuzzy closure spaces, and let $(X,c) = \big(\underset{t \in T}{\prod} X_t, \otimes c_t\big)$ denote their product. For $x \in \underset{t \in T}{\prod} X_t$, let $x^t$ denote its $t^{\text{th}}$ coordinate. Then, for any fuzzy point $x_\lambda$, its closure in the product space is given by
				\begin{center}
					$
					\otimes c_t(x_\lambda) = \underset{t \in T}{\prod} c_t(x^t_\lambda).$
					
				\end{center}
			\end{lemma}
			
			\begin{proof}
				Suppose $y_\gamma \in \otimes c_t(x_\lambda)$. By Definition \ref{thm:prod}, this implies that for every $t \in T$ we have
				\[
				P_t(y_\gamma) = y^t_\gamma \in c_t(x^t_\lambda).
				\]
				Hence,
				\begin{center}
					$y_\gamma \in \underset{t \in T}{\prod} c_t(x^t_\lambda).$
				\end{center}
				
				Conversely, assume $z_\rho \in \underset{t \in T}{\prod} c_t(x^t_\lambda)$. Then, for each $t \in T$, we have
				\[
				z_\rho^t \in c_t(x^t_\lambda).
				\]
				If $x_\lambda$ can be expressed as a finite join,
				$
				x_\lambda = f_1 \vee f_2 \vee \cdots \vee f_n,	$
				then, since the join is finite, at least one $f_i = x_\lambda$. Consequently,
				\[
				P_t(z_\rho) = z_\rho^t \in c_t(x^t_\lambda) = c_t(P_t(f_i)) \quad \text{for all } t \in T.
				\]
				Therefore,
				\[
				z_\rho \in \otimes c_t(x_\lambda).
				\]
				
				Thus, we conclude that
				\begin{center}
					$\otimes c_t(x_\lambda) = \underset{t \in T}{\prod} c_t(x^t_\lambda).$ \quad \quad  \qedhere
				\end{center}
			\end{proof}

			\begin{Theorem}
				Let $\mathcal{F}= \{(X_t,c_t): t \in T\}$ be a family of fuzzy closure spaces. Then the product space $(X,c) = \big(\underset{t \in T}{\prod}X_t, \otimes c_t\big)$ is \v{C}F$T_0$ whenever each factor space $(X_t,c_t)$ is \v{C}F$T_0$.
			\end{Theorem}
			
			\begin{proof}
				Assume that each $(X_t,c_t)$ is \v{C}F$T_0$. Suppose, for the sake of contradiction, that the product space $(X,c)$ is not \v{C}F$T_0$. Then there exist distinct fuzzy points $x_\lambda, y_\gamma \in X$ such that
				\[
				x_\lambda \notin \mathfrak{Co}_X(\otimes c_t(y_\gamma)) \quad \text{and} \quad y_\gamma \notin \mathfrak{Co}_X(\otimes c_t(x_\lambda)).
				\]
				
				This implies that
				\[
				\lambda > 1 - \otimes c_t(y_\gamma)(x),
				\]
				and hence
				\[
				\otimes c_t(y_\gamma)(x) > 1 - \lambda.
				\]
				From Lemma \ref{thm:prod2}, we have $
				\otimes c_t(y_\gamma)= \underset{t \in T}{\prod} c_t(y^t_\gamma),$
				which implies that
				\[
				c_t(y^t_\gamma)(x^t) > 1 - \lambda \quad \text{for all } t \in T.
				\]
				Equivalently,
				\[
				\lambda > 1 - c_t(y^t_\gamma)(x^t) \quad \text{for all } t \in T,
				\]
				which yields
				\[
				x^t_\lambda \notin \mathfrak{Co}_{X_t}(c_t(y^t_\gamma)) \quad \text{for all } t \in T.
				\]
				By a similar argument, we also obtain
				\[
				y^t_\gamma \notin \mathfrak{Co}_{X_t}(c_t(x^t_\lambda)) \quad \text{for all } t \in T.
				\]
				
				This contradicts the assumption that each $(X_t,c_t)$ is \v{C}F$T_0$. Therefore, the assumption that $(X,c)$ is not \v{C}F$T_0$ must be false. Hence, the product space $\big(\underset{t \in T}{\prod}X_t, \otimes c_t\big)$ is \v{C}F$T_0$.
			\end{proof}

			\section{\v{C}F$T_1$ fuzzy closure spaces}
			\vspace{6pt}
			\begin{Definition}
				A fuzzy closure space is \v{C}F$T_1$ if for every two distinct fuzzy points $x_\lambda, y_{\gamma} \in I^X$,  $x_\lambda \in \mathfrak{Co}( c(y_\gamma))$ and $ y_\gamma \in \mathfrak{Co} (c(x_\lambda))$.
			\end{Definition}
			Clearly, every $FT_1$ fuzzy topological space can be viewed as a \v{C}F$T_1$ fuzzy closure space.

			\begin{remark}
				The \v{C}F$T_1$ separation property is hereditary. Additionally, if $c_1$ and $c_2$ are fuzzy closure operators on $X$ satisfying $c_1 \leq c_2$, then $(X,c_1)$ being \v{C}F$T_1$ implies that $(X,c_2)$ is also \v{C}F$T_1$. 
			\end{remark}


			\begin{Theorem}\label{thm:CFT1}
				A fuzzy closure space $(X,c)$ is \v{C}F$T_1$ if and only if every fuzzy point is well closed.
			\end{Theorem}
			
			\begin{proof}
				$(\Rightarrow)$ Assume that $(X,c)$ is \v{C}F$T_1$. Suppose, to the contrary, that there exists a fuzzy point $x_\lambda \in I^X$ which is not well closed. Then, there exists some $y \in X$, with $y \neq x$, such that $y_\gamma \in c(x_\lambda)$ for some $\gamma \in (0,1]$. Consider the fuzzy points $x_\lambda$ and $y_1$. Clearly, $y_1 \notin \mathfrak{Co}(c(x_\lambda))$, which contradicts the assumption that $(X,c)$ is \v{C}F$T_1$. Hence, the assumption is false, and every fuzzy point $x_\lambda$ must be well closed.
				
				$(\Leftarrow)$ Conversely, assume that every fuzzy point is well closed. Then, for all distinct $x,y \in X$ and for all $0<\lambda,\gamma \leq 1$, we have $x_\lambda \notin c(y_\gamma)$. Consequently, $x_\lambda \in \mathfrak{Co}(c(y_\gamma))$ and $y_\gamma \in \mathfrak{Co}(c(x_\lambda))$. Therefore, $(X,c)$ is a \v{C}F$T_1$ fcs.
			\end{proof}

			\begin{lemma}\label{thm:CFT1L}
				In a fuzzy closure space $(X,c)$, every fuzzy point is well closed if and only if every fuzzy singleton is closed.
			\end{lemma}
			
			\begin{proof}
				$(\Rightarrow)$ Assume that every fuzzy point in $(X,c)$ is well closed. In particular, each fuzzy singleton $x_1 \in X$ is well closed, which implies $c(x_1)=x_1$.  
				
				$(\Leftarrow)$ Conversely, suppose $c(x_1)=x_1$ for every $x \in X$. Then, for any fuzzy point $x_\lambda$ with $\lambda \in (0,1]$, we have $c(x_\lambda) \leq c(x_1) = x_1,$
				which shows that $x_\lambda$ is well closed.  
			\end{proof}

			\begin{Theorem}\label{thm:T1}
				A fuzzy closure space $(X,c)$ is \v{C}F$T_1$ if and only if every fuzzy singleton in $X$ is closed.
			\end{Theorem}
			\begin{proof}
				The result follows immediately from Theorem \ref{thm:CFT1} and Lemma \ref{thm:CFT1L}.
			\end{proof}
			
			As a direct consequence of the above characterization, we obtain the following result.  
			
			\begin{Theorem}
				A fuzzy closure space $(X,c)$ is \v{C}F$T_1$ if and only if its associated fuzzy topological space $(X,\tau(c))$ is $FT_1$.
			\end{Theorem}
			\begin{proof}
				This statement follows directly from Theorem \ref{thm:FT1} and Theorem \ref{thm:T1}.
			\end{proof}
			
			\begin{Definition}
				A fuzzy closure space is said to be \v{C}F$T_s$ (or strongly \v{C}F$T_1$) if every fuzzy point is closed.
			\end{Definition}
			
			It is immediate that every \v{C}F$T_s$ space is also \v{C}F$T_1$. In the setting of \v{C}ech closure spaces, every finite $T_1$ space is discrete. An analogous result in the context of fuzzy closure spaces is stated below.  
			
			\begin{Theorem}
				Every finite \v{C}F$T_s$ fuzzy closure space is discrete.
			\end{Theorem}
			\begin{proof}
				Straightforward.
			\end{proof}
			
			\begin{remark}
				Every \v{C}F$T_1$ fuzzy closure space is \v{C}F$T_0$.  
				However, the converse does not hold in general, as demonstrated by the following counterexample.
				
				Consider the set of natural numbers $\mathbb{N}$ equipped with the fuzzy closure operator $c$ defined by $
				c(\underline{0}) = \underline{0}, \quad 
				c(x_\lambda) = x_\lambda \vee (x+1)_\lambda, \text{ and }
				c(f) = \bigvee_{x_\lambda \in f} c(x_\lambda),$
				for other fuzzy subsets $f$ of $\mathbb{N}$.
				
				If $y = x+1$, then
				\[
				x_\lambda \in \mathfrak{Co}(c(x_\lambda))(y_\gamma) = 1_{\mathbb{N}\setminus \{x+1, x+2\}} \vee (x+1)_{1-\gamma} \vee (x+2)_{1-\gamma}.
				\]
				On the other hand, if $y \notin \{ x, x+1\}$, we obtain
				\[
				y_\gamma \in \mathfrak{Co}(c(x_\lambda)) = 1_{\mathbb{N}\setminus \{x, x+1\}} \vee (x)_{1-\lambda} \vee (x+1)_{1-\lambda}.
				\]
				
				From this, it follows that $(\mathbb{N},c)$ satisfies the $T_0$ condition. However, every fuzzy point $x_\lambda$ in $(\mathbb{N},c)$ is not well-closed. Therefore, $(\mathbb{N},c)$ is not \v{C}F$T_1$.
			\end{remark}

			\begin{Theorem}
				The property \v{C}F$T_1$ is a fuzzy closure property.
			\end{Theorem}
			
			\begin{proof}
				Let $(X,c)$ and $(Y,d)$ be two fuzzy closure spaces, and let $
				\theta : (X,c) \longrightarrow (Y,d)$
				be a \v{C}F-homeomorphism. Suppose $y_1$ is a fuzzy singleton in $Y$. Then there exists an $x \in X$ such that $\theta(x_1) = y_1.$
				Since $\theta$ is a homeomorphism, we have
				\[
				y_1 = \theta(x_1) = \theta(c(x_1)) = d(\theta(x_1)) = d(y_1).
				\]
				Thus, $y_1$ is closed in $(Y,d)$. Hence, if $(X,c)$ is \v{C}F$T_1$, then $(Y,d)$ is also \v{C}F$T_1$. Therefore, \v{C}F$T_1$ is preserved under \v{C}F-homeomorphisms, and consequently it is a fuzzy closure property.
			\end{proof}

			\begin{Theorem}
				A fuzzy closure space $(X,c)$ is \v{C}F$T_1$ if and only if for every pair of distinct fuzzy points $x_\lambda$ and $y_\gamma$, there exists $f,g\in I^X$ such that  $x_\lambda \in \intt(f) \text{ and } y_\gamma \in \mathfrak{Co}(f)$ and $y_\gamma \in \intt(g) \text{ and } x_\lambda \in \mathfrak{Co}(g)$.
			\end{Theorem}
			\begin{proof}
				Similar proof as in the case of \v{C}F$T_0$ fcs.
			\end{proof}
			In the following theorems, we prove additive and productive properties of the \v{C}F$T_1$ separation axiom in fuzzy closure spaces.

			\begin{Theorem}
				Let $\mathcal{F} = \{(X_t, c_t) : t \in T\}$ be a family of pairwise disjoint fuzzy closure spaces. Then their sum
				$(\underset{t \in T}{\bigvee} X_t, \ \oplus c_t)$
				is \v{C}F$T_1$ if and only if each $(X_t, c_t)$ is \v{C}F$T_1$.
			\end{Theorem}
			
			\begin{proof}
				$(\Rightarrow)$ Suppose that $(\underset{t \in T}{\bigvee} X_t, \oplus c_t)$ is \v{C}F$T_1$.  
				Let $x_1$ be a fuzzy singleton in $X_{t_1}$, for some $t_1 \in T$.  
				Since $(\bigvee_{t \in T} X_t, \oplus c_t)$ is \v{C}F$T_1$, treating $x_1$ as a fuzzy singleton of $\underset{t \in T}{\bigvee} X_t$, we have $\oplus c_t(x_1) = x_1.$\\
				By the definition of the sum closure operator,
				\[
				\oplus c_t(x_1) = \underset{t \in T}{\bigvee} c_t(x_1 \wedge 1_{X_t}) = c_{t_1}(x_1).
				\]
				Thus, $c_{t_1}(x_1) = x_1$, which shows that $(X_{t_1}, c_{t_1})$ is \v{C}F$T_1$.  
				Since $t_1 \in T$ was arbitrary, it follows that each $(X_t, c_t)$ is \v{C}F$T_1$.
				
				\medskip
				$(\Leftarrow)$ Conversely, suppose that each $(X_t, c_t)$ is \v{C}F$T_1$ for all $t \in T$.  
				Let $x_1$ be a fuzzy singleton in $\underset{t \in T}{\bigvee} X_t$.  
				Then there exists a unique $t_1 \in T$ such that $x \in X_{t_1}$.  
				Since $(X_{t_1}, c_{t_1})$ is \v{C}F$T_1$, we have
				\[
				c_{t_1}(x_1) = x_1.
				\]
				Therefore,
				\[
				\oplus c_t(x_1) = \underset{t \in T}{\bigvee} c_t(x_1 \wedge 1_{X_t}) = c_{t_1}(x_1) = x_1.
				\]
				Hence, the sum space $(\underset{t \in T}{\bigvee} X_t, \oplus c_t)$ is \v{C}F$T_1$.
			\end{proof}

			\begin{Theorem}
				Let $\mathcal{F} = \{(X_t, c_t) : t \in T\}$ be a family of fuzzy closure spaces. Then their product
				$(X, c) = (\underset{t \in T}{\prod}X_t, \ \otimes c_t )$ is \v{C}F$T_1$ if and only if each factor space $(X_t, c_t)$ is \v{C}F$T_1$.
			\end{Theorem}
			
			\begin{proof}
				$(\Rightarrow)$ Suppose that each $(X_t, c_t)$ is \v{C}F$T_1$.  
				Let $x_\lambda \in I^X$. We must show that $x_\lambda$ is well closed in $(X, c)$.  
				If $y_\gamma \in \otimes c_t(x_\lambda)$, then by the definition of the product fuzzy closure, this means
				\[
				P_t(y_\gamma) \in c_t(P_t(x_\lambda)) 
				\quad \text{for all } t \in T,
				\]
				that is,
				\[
				y^t_\gamma \in c_t(x^t_\lambda) \quad \text{for all } t \in T.
				\]
				Since $(X_t, c_t)$ is \v{C}F$T_1$, it follows that,
				\[
				y^t_\gamma \in c_t(x^t_\lambda) \leq x^t_1,
				\]
				which implies $y^t = x^t$ for all $t \in T$. Hence, $y = x$, and therefore $x_\lambda$ is well closed.  
				Thus, $(X, c)$ is \v{C}F$T_1$.
				
				\medskip
				$(\Leftarrow)$ Conversely, assume that $(X, c) = (\underset{t \in T}{\prod}X_t, \ \otimes c_t )$ is \v{C}F$T_1$.  
				Suppose, for contradiction, that there exists $t_1 \in T$ such that $(X_{t_1}, c_{t_1})$ is not \v{C}F$T_1$. Then there exists a fuzzy singleton $x^{t_1}_1$ in $X_{t_1}$ that is not closed. Consequently, there exists $y^{t_1} \neq x^{t_1}$ in $X_{t_1}$ such that
				\[
				y^{t_1}_\gamma \in c_{t_1}(x^{t_1}_1) 
				\quad \text{for some } \gamma \in (0,1].
				\]
				
				Now, for each $t \neq t_1$, choose $x^t, y^t \in X_t$ such that $x^t = y^t$.  
				Consider the fuzzy points $x_1$ and $y_\gamma$ in $X = \underset{t \in T}{\prod}X_t$.  
				Then, for each $t \neq t_1$,
				\[
				P_t(y_\gamma) = y^t_\gamma = x^t_\gamma \in c_t(x^t_1),
				\]
				and for $t = t_1$, by assumption, 
				\[
				y^{t_1}_\gamma \in c_{t_1}(x^{t_1}_1).
				\]
				Therefore,
				\begin{center}
					
					$y_\gamma \in \underset{t \in T}{\prod} c_t(x^t_1) = \otimes c_t(x_1),$
				\end{center}
				which implies that $x_1$ is not closed in $(X, c)$.  
				
				This contradicts the assumption that $(X, c)$ is \v{C}F$T_1$.  
				Hence, $(X_t, c_t)$ is \v{C}F$T_1$ for all $t \in T$.
			\end{proof}

			\section{\v{C}F$T_2$ fuzzy closure spaces}
			
			\vspace{6pt}
			\begin{Definition}
				A fuzzy closure space $(X,c)$ is said to be \v{C}F$T_2$ or \v{C}ech fuzzy Hausdorff if for every pair of distinct fuzzy points $x_\lambda$ and $ y_\gamma $ in $X$, there exist two neighborhoods $f$ and $g$ of $x_\lambda$ and $ y_\gamma $ respectively such that  $f\leq \mathfrak{Co}(g)$ and $x_\lambda \in f\leq \mathfrak{Co}(y_\gamma) $ and $ y_\gamma \in g \leq \mathfrak{Co}(x_\lambda) $.
			\end{Definition}
			It is evident that every $FT_2$ fuzzy topological space is a \v{C}F$T_2$ fuzzy closure space. Let $c_1\text{ and } c_2$ are two fuzzy closure operators on $X$ such that $c_1\leq c_2$; if $(X,c_1)$ is \v{C}F$T_2$, then $(X,c_2)$ is also \v{C}F$T_2$.
			
			\begin{Theorem}
				If $(X, \tau(c))$ is $FT_2$ fts, then $(X,c)$ is \v{C}F$T_2$ fcs.
			\end{Theorem}
			\begin{proof}
				The proof is trivial since every open set containing $x_\lambda$ in  $(X, \tau(c))$ is a neighborhood of $x_\lambda$ in $(X,c)$. 
			\end{proof}

			\begin{Theorem}
				Every \v{C}F$T_2$ fuzzy closure space is \v{C}F$T_1$.
			\end{Theorem}
			
			\begin{proof}
				Let $(X, c)$ be a \v{C}F$T_2$ fuzzy closure space, and let $x_\lambda$ and $y_\gamma$ be two distinct fuzzy points in $X$.  
				By the \v{C}F$T_2$ property, there exist neighborhoods $f$ of $x_\lambda$ and $g$ of $y_\gamma$ such that 
				\[
				x_\lambda \in f \leq \mathfrak{Co}(y_\gamma), 
				\quad 
				y_\gamma \in g \leq \mathfrak{Co}(x_\lambda),
				\quad \text{and} \quad 
				f \leq \mathfrak{Co}(g).
				\]
				Since $x_\lambda \in f \leq \operatorname{int}(f)$, we obtain
				\[
				x_\lambda \in \operatorname{int}(f) \leq \operatorname{int}(\mathfrak{Co}(y_\gamma)).
				\]
				By the definition of interior,
				\[
				\operatorname{int}(\mathfrak{Co}(y_\gamma)) 
				= \mathfrak{Co}\!\bigl(c(\mathfrak{Co}(\mathfrak{Co}(y_\gamma)))\bigr) 
				= \mathfrak{Co}(c(y_\gamma)).
				\]
				Thus,
				\[
				x_\lambda \in \mathfrak{Co}(c(y_\gamma)).
				\]
				
				Similarly, $y_\gamma \in \mathfrak{Co}(c(x_\lambda))$. Therefore, $(X, c)$ is \v{C}F$T_1$.  
			\end{proof}

			\begin{Theorem}
				Every finite \v{C}F$T_1$ fuzzy closure space is \v{C}F$T_2$.
			\end{Theorem}
			
			\begin{proof}
				Let $(X, c)$ be a \v{C}F$T_1$ fuzzy closure space, where $X = \{x^1, x^2, x^3, \dots, x^n\}$ is a finite set. By the \v{C}F$T_1$ property, each fuzzy singleton $x^i_1$ $(i = 1,2,\dots,n)$ is closed. Since here the complement of a fuzzy singleton is also closed, it follows that every fuzzy singleton is open as well.  
				
				Now, let $x^i_\lambda$ and $x^j_\gamma$ be two distinct fuzzy points of $X$. Define
				\[
				f = x^i_1 \quad \text{and} \quad g = x^j_1.
				\]
				Since $f$ and $g$ are fuzzy singletons, we have 
				\[
				\operatorname{int}(f) = f \quad \text{and} \quad \operatorname{int}(g) = g.
				\]
				Therefore,
				\[
				x^i_\lambda \in f \leq \mathfrak{Co}(x^j_\gamma),
				\quad x^j_\gamma \in g \leq \mathfrak{Co}(x^i_\lambda),\quad \text{and} \quad f \leq \mathfrak{Co}(g).
				\]
				
				Hence, $f$ and $g$ form disjoint neighborhoods that separate $x^i_\lambda$ and $x^j_\gamma$, demonstrating that $(X, c)$ satisfies the \v{C}F$T_2$ condition.  
			\end{proof}

			But in general, every \v{C}F$T_1$ space is not \v{C}F$T_2$. For, consider the example below.

			\begin{example}
				Let $X$ be an infinite set, and define a fuzzy closure operator $c$ on $X$ as follows: 
				
				\[ \text{if } \operatorname{supp}(f) \text{ is finite},\quad
				c(f)(x) =
				\begin{cases}
					0, & \text{if } f(x) = 0, \\[6pt]
					f(x) + \dfrac{1}{2}, & \text{if } 0 < f(x) < \dfrac{1}{2},   \\[6pt]
					1, & \text{otherwise},
				\end{cases} 
				\]
				
				\[
				\text{if } \operatorname{supp}(f) \text{ is infinite}, \quad c(f) = \underline{1}.
				\]
				
				It is straightforward to verify that $c$ is a fuzzy closure operator, but not a fuzzy topological closure operator.  
				Since every fuzzy singleton is closed, it follows that $(X, c)$ is \v{C}F$T_1$.  
				We now examine the interiors of fuzzy subsets of $X$:  
				
				\medskip
				{Case 1.} If $\operatorname{supp}(\mathfrak{Co}(f))$ is infinite, then
				$
				c(\mathfrak{Co}(f)) = \underline{0},
				$
				and hence
				\[
				\operatorname{int}(f) = \mathfrak{Co}(c(\mathfrak{Co}(f))) = \underline{0}.
				\]
				
				\medskip
				{Case 2.} If $f$ takes the value $1$ at all but finitely many points of $X$.  
				Let $g = \underset{x_1 \in f}{\bigvee} x_1.$ Then $g \leq f$, and
				\[
				\operatorname{int}(g) = \mathfrak{Co}(c(\mathfrak{Co}(g))) 
				= \mathfrak{Co}(\mathfrak{Co}(g)) 
				= g.
				\]
				Thus,
				\[
				\operatorname{int}(f) \geq \operatorname{int}(g) = g = \bigvee_{x_1 \in f} x_1.
				\]
				
				\medskip
				From these observations, it follows that we cannot obtain two fuzzy subsets $f$ and $g$ of $X$ with non-empty interiors such that $f \leq \mathfrak{Co}(g)$. Therefore, $(X, c)$ is not \v{C}F$T_2$.
			\end{example}

			\begin{remark}
				Every subspace of a \v{C}F$T_2$ fuzzy closure space is \v{C}F$T_2$. (i.e., \v{C}F$T_2$ is  a hereditary property.)
			\end{remark}

			\begin{lemma}\label{thm:interior}
				If $\theta : (X, c) \to (Y, d)$ is a \v{C}F-homeomorphism, then for every fuzzy subset $f$ of $X$, we have
				\[
				\theta\!\bigl(\operatorname{int}_X(f)\bigr) = \operatorname{int}_Y\!\bigl(\theta(f)\bigr).
				\]
			\end{lemma}
			
			\begin{proof}
				Since $\theta$ is a \v{C}F-homeomorphism, we have,
				\[
				\theta(c(f)) = d(\theta(f)) 
				\quad \text{and} \quad 
				\theta(\mathfrak{Co}_X(f)) = \mathfrak{Co}_Y(\theta(f)).
				\]
				Using these identities, we compute:
				\begin{align*}
					\operatorname{int}_Y(\theta(f)) 
					&= \mathfrak{Co}_Y\!\bigl(d(\mathfrak{Co}_Y(\theta(f)))\bigr) \\[6pt]
					&= \mathfrak{Co}_Y\!\bigl(d(\theta(\mathfrak{Co}_X(f)))\bigr) \\[6pt]
					&= \mathfrak{Co}_Y\!\bigl(\theta(c(\mathfrak{Co}_X(f)))\bigr) \\[6pt]
					&= \theta\!\bigl(\mathfrak{Co}_X(c(\mathfrak{Co}_X(f)))\bigr) \\[6pt]
					&= \theta\!\bigl(\operatorname{int}_X(f)\bigr).   \quad \quad \quad \qedhere
				\end{align*} 
			\end{proof}

			\begin{Theorem}
				\v{C}F$T_2$ is a fuzzy closure property.
			\end{Theorem}
			
			\begin{proof}
				Let $(X, c)$ and $(Y, d)$ be fuzzy closure spaces (fcs’s), and let $\theta : (X, c) \longrightarrow (Y, d)$ be a \v{C}F-homeomorphism. Consider two distinct fuzzy points $y^1_\lambda, y^2_\gamma \in Y$. Since $\theta$ is bijective, there exist distinct points $x^1, x^2 \in X$ such that 
				\[
				\theta(x^1) = y^1 
				\quad \text{and} \quad 
				\theta(x^2) = y^2.
				\]
				
				By assumption, $(X, c)$ is \v{C}F$T_2$. Hence, there exist neighborhoods $f_1$ and $f_2$ of $x^1_\lambda$ and $x^2_\gamma$, respectively, which separate these fuzzy points. By the preceding lemma, the images $\theta(f_1)$ and $\theta(f_2)$ are neighborhoods of $y^1_\lambda$ and $y^2_\gamma$, respectively, and they separate $y^1_\lambda$ and $y^2_\gamma$ in $(Y, d)$. 
				
				Thus, $(Y, d)$ is \v{C}F$T_2$. Therefore, the property \v{C}F$T_2$ is preserved under \v{C}F-homeomorphisms, and hence it is a fuzzy closure property.
			\end{proof}

			\begin{lemma}
				Let $\mathcal{F}= \{(X_t,c_t): t \in  T\}$ be a family of pairwise disjoint fuzzy closure spaces and let $(X,c)=(\underset{t \in T}{\bigvee} X_t,\oplus c_t)$ be their sum. Then, $\intt_{X} (f)= \underset{t \in T}{\bigvee} \intt_{X_t} (f \wedge 1_{X_t})$.
			\end{lemma}
			\begin{proof}
				By definition, $\oplus c_t (f)=  \underset{t \in T}{\bigvee} c_t(f \wedge 1_{X_t})$. Thus,
				\begin{align*}
					\intt_{X} (f) &=  \mathfrak{Co}_X(\oplus c_t (\mathfrak{Co}_X(f)))\\
					&= \mathfrak{Co}_X \Big(\underset{t \in T}{\bigvee} c_t (\mathfrak{Co}_X(f) \wedge 1_{X_t})\Big)\\
					&= \mathfrak{Co}_X \Big(\underset{t \in T}{\bigvee} c_t \big(\mathfrak{Co}_{X_t}(f \wedge 1_{X_t})\big)\Big)\\
					&=\underset{t \in T}{\bigvee} \mathfrak{Co}_{X_t} \Big( c_t \big(\mathfrak{Co}_{X_t}(f \wedge 1_{X_t})\big) \Big)\\
					&= \underset{t \in T}{\bigvee}  \intt_{X_t}((f \wedge 1_{X_t}) &&\qedhere 
				\end{align*}
			\end{proof}

			\begin{Theorem}
				Let $\mathcal{F} = \{(X_t, c_t) : t \in T\}$ be a family of pairwise disjoint fuzzy closure spaces. Then the sum $(X, c) = (\underset{t \in T}{\bigvee} X_t, \, \oplus c_t)$	is \v{C}F$T_2$ if and only if each $(X_t, c_t)$ is \v{C}F$T_2$.
			\end{Theorem}
			
			\begin{proof}
				$(\Rightarrow)$ Assume that $(X, c)$ is \v{C}F$T_2$.  
				Let $x_\lambda$ and $y_\gamma$ be two distinct fuzzy points in $X_s$ for some $s \in T$. Considering $x_\lambda$ and $y_\gamma$ as fuzzy points of the \v{C}F$T_2$ fuzzy closure space $(X, c)$, there exist neighborhoods $f, g \in I^X$ of $x_\lambda$ and $y_\gamma$, respectively, such that 
				\[
				f \leq \mathfrak{Co}(g), \quad x_\lambda \in f \leq \mathfrak{Co}(y_\gamma), \quad \text{and} \quad y_\gamma \in g \leq \mathfrak{Co}(x_\lambda).
				\]
				Consider the fuzzy sets $f \wedge 1_{X_s}$ and $g \wedge 1_{X_s}$. Since
				\[
				x_\lambda \in \operatorname{int}_X(f) = \bigvee_{t \in T} \operatorname{int}_{X_t}(f \wedge 1_{X_t}),
				\]
				we deduce that 
				\[
				x_\lambda \in \operatorname{int}_{X_s}(f \wedge 1_{X_s}),
				\]
				and similarly,
				\[
				y_\gamma \in \operatorname{int}_{X_s}(g \wedge 1_{X_s}).
				\]
				Thus, the neighborhoods $f \wedge 1_{X_s}$ and $g \wedge 1_{X_s}$ separate $x_\lambda$ and $y_\gamma$ in $(X_s, c_s)$. Hence, $(X_s, c_s)$ is \v{C}F$T_2$ for all $s \in T$.
				
				\medskip
				$(\Leftarrow)$ Conversely, suppose that $(X_t, c_t)$ is \v{C}F$T_2$ for all $t \in T$. Let $x_\lambda, y_\gamma \in I^X$ be distinct fuzzy points.  
				
				{Case (1):} If $x, y \in X_{t_1}$ for some $t_1 \in T$, then there exist disjoint neighborhoods $f, g \in I^{X_{t_1}}$ containing $x_\lambda$ and $y_\gamma$, respectively. Considering $f$ and $g$ as fuzzy subsets of $X$, and by the preceding lemma, we have
				\[
				\operatorname{int}_X(f) = \operatorname{int}_{X_{t_1}}(f) 
				\quad \text{and} \quad 
				\operatorname{int}_X(g) = \operatorname{int}_{X_{t_1}}(g).
				\]
				Thus, $f$ and $g$ also separate $x_\lambda$ and $y_\gamma$ in $\bigl(\underset{t \in T}{\bigvee} X_t, \oplus c_t\bigr)$.
				
				{Case (2):} If $x \in X_{t_1}$ and $y \in X_{t_2}$ with $t_1 \neq t_2$, then since
				\[
				\operatorname{int}_X(1_{X_{t_1}}) = 1_{X_{t_1}} 
				\quad \text{and} \quad
				\operatorname{int}_X(1_{X_{t_2}}) = 1_{X_{t_2}},
				\]
				we may take $1_{X_{t_1}}$ and $1_{X_{t_2}}$ as neighborhoods of $x_\lambda$ and $y_\gamma$, respectively, in $(\underset{t \in T}{\bigvee} X_t, \oplus c_t)$. These clearly separate $x_\lambda$ and $y_\gamma$.
				
				Hence, $(\underset{t \in T}{\bigvee} X_t, \oplus c_t)$ is \v{C}F$T_2$.
			\end{proof}

			\textbf{\subsection{\v{C}F$T_{2\frac{1}{2}}$ fuzzy closure spaces}}
			\vspace{6pt}

			\begin{Definition}
				A fuzzy closure space $(X,c)$ is said to be \v{C}F$T_{2\frac{1}{2}}$ or \v{C}F-Urysohn if and only if for every pair of distinct fuzzy points $x_\lambda$ and $y_\gamma$, there exist two neighborhoods $f$ and $g$ of $x_\lambda$ and $ y_\gamma $ respectively such that $c(f)\leq \mathfrak{Co}(c(g))$ and $x_\lambda \in f\leq \mathfrak{Co}(y_\gamma) $ and $ y_\gamma \in g \leq \mathfrak{Co}(x_\lambda) $. 
			\end{Definition}
			
			Every \v{C}F-Urysohn fcs is \v{C}F$T_2$ fcs. If $(X, \tau(c))$ is fuzzy Urysohn, then $(X,c)$ is \v{C}F-Urysohn. Since a fuzzy closure space is \v{C}F$T_1$ iff every fuzzy singleton is closed, we can easily get that every finite \v{C}F$T_1$ space is \v{C}$FT_{2\frac{1}{2}}$. Also we can  find that \v{C}$FT_{2\frac{1}{2}}$ is a hereditary property.

			\begin{Theorem}
				\v{C}F$T_{2\frac{1}{2}}$ is a fuzzy closure property.
			\end{Theorem}
			\begin{proof}
				Similar proof as that of \v{C}F$T_2$ spaces.
			\end{proof}

			\begin{Theorem}
				Let $\mathcal{F}= \{(X_t,c_t): t \in  T\}$ be a family of  pairwise disjoint fuzzy closure spaces. The sum  $(X,c)= (\underset{t\in T} \bigvee X_t, \oplus c_t)$ is a \v{C}F$T_{2\frac{1}{2}}$ space if and only if each $(X_t,c_t)$ is a \v{C}F$T_{2\frac{1}{2}}$ space.
			\end{Theorem}
			\begin{proof}
				Similar proof as that of the sum of \v{C}F$T_2$ spaces.
			\end{proof}

			\section{\v{C}F-regular fuzzy closure spaces}
			\vspace{6pt}

			\begin{Definition}
				A fuzzy closure space $(X,c)$ is said to be \v{C}F-regular if for each fuzzy point $x_\lambda$ in $X$ and each non-empty fuzzy subset $k$ of $X$ such that $x_\lambda \in \mathfrak{Co}(c(k))$, there exists neighborhoods $f$ of $x_\lambda$  and $g$ of $k$ such that $ f \leq \mathfrak{Co}(g)$. A \v{C}F-regular fcs which is also $T_s$ is said to be \v{C}F$T_3$. 
			\end{Definition}

			The indiscrete fuzzy closure operator is \v{C}F-regular. Consequently, \v{C}F-regularity does not imply \v{C}F$T_{2\frac{1}{2}}$ property or any of the weaker separation properties in the hierarchy below it. However, it can be readily verified that every \v{C}F$T_3$ space
			is also a \v{C}F$T_2$ space. Furthermore, if $c$ is a fuzzy topological closure operator, then the fts $(X,\tau(c))$ is fuzzy regular ($FT_3$) if and only if the fcs $(X,c)$ is \v{C}F-regular (\v{C}F$T_3$).
			
			\begin{remark}
				The fts $(X,\tau(c))$ being fuzzy regular does not necessarily imply that $(X,c)$ is \v{C}F-regular fcs. This is demonstrated in the following example.
			\end{remark}
			
			\begin{example}{\label{thm:regular}}
				Consider the set $\mathbb{Z}$ equipped with a fuzzy closure operator $c$ defined as follows: $c(\underline{0}) = \underline{0}$, $c(x_\lambda) = x_\lambda \vee (x+1)_\lambda$ for any fuzzy point $x_\lambda \in I^{\mathbb{Z}}$, and for any fuzzy set $f \in I^{\mathbb{Z}}$, $c(f) = \underset{x_\lambda \leq f}{\bigvee} c(x_\lambda)$. The associated \fts\ $(\mathbb{Z}, \tau(c))$ is indiscrete and is fuzzy regular. 
				
				Let $x \in \mathbb{Z}$, and consider the fuzzy point $x_\lambda$ and the fuzzy set $k = (x+1)_1$. It can be observed that every neighborhood of the fuzzy point $(x+1)_1$ contains the fuzzy set $x_1 \vee (x+1)_1$. Since $x_\lambda \in Co(c(k))$, and every neighborhood of $k$ contains $x_1$, it follows that $x_\lambda$ and $k$ cannot be separated. Therefore, $(X,c)$ is not a {\vC}F regular \fcs.
			\end{example}

			\begin{remark}
				Every \v{C}F-Urysohn (or \v{C}F$T_2$) space need not be \v{C}F-regular. This can be seen from the following example.
			\end{remark}

			\begin{example} \label{thm:nonex}
				Fix an element $x \in \mathbb{N}$ and define a finitely generated fuzzy closure operator $c$ on $\mathbb{N}$ as follows:
				\[
				c(f) = 
				\begin{cases}
					\underline{0} & \text{if } f = \underline{0}, \\
					x_{\lambda + \frac{1}{2}} & \text{if } f = x_\lambda \text{ and } \lambda \in (0, \frac{1}{2}), \\
					x_1 & \text{if } f = x_\lambda \text{ and } \lambda \in [\frac{1}{2}, 1], \\
					y_\lambda & \text{if } f = y_\lambda \text{ and } y \neq x, \\
					\underset{y_\lambda \leq f}{\bigvee} c(y_\lambda) & \text{otherwise}.
				\end{cases}
				\]
				It can be verified that $c$ is a fuzzy closure operator but not fuzzy topological. Since every fuzzy singleton is closed, $(\mathbb{N}, c)$ is \v{C}F$T_1$. For any two distinct fuzzy points $y_\lambda, z_\gamma \in I^{\mathbb{N}}$, the fuzzy sets $f = y_1$ and $g = z_1$ separate these points, and $c(f) \leq \Co(c(g))$. Hence, $(\mathbb{N}, c)$ is \vC$F$-Urysohn (and also \vC$FT_2$).
				
				If $\lambda > \frac{1}{2}$, then
				\begin{align*}
					\intt(x_\lambda) &= \Co(c(\Co(x_\lambda))) \\
					&= \Co(c(1_{\{\mathbb{N} \setminus \{x\}\}} \vee x_{1 - \lambda})) \\
					&= \Co(1_{\{\mathbb{N} \setminus \{x\}\}} \vee x_{\frac{3}{2} - \lambda}) \\
					&= x_{\lambda - \frac{1}{2}}.
				\end{align*}
				Similarly, if $\lambda \leq \frac{1}{2}$, then $\operatorname{int}(x_\lambda) = \underline{0}$. Consider the fuzzy point $x_{0.4}$ and the fuzzy set $k = x_{0.05}$. Clearly, $x_{0.4} \in \Co(c(k)) = \underline{1}_{\mathbb{N} \setminus \{x\}} \vee x_{0.45}$. Every neighborhood $f$ of $x_{0.4}$ contains $x_{0.9}$, and every neighborhood $g$ of $x_{0.05}$ contains $x_{0.55}$, so $f \nleq \Co(g)$. Therefore, $(\mathbb{N}, c)$ is not \vC F-regular, and hence not \v{C}F$T_3$.
			\end{example}

			\begin{remark}
				It is important to note that not every \v{C}F-regular space is necessarily a \v{C}F$T_3$ space. 
				
				Consider a set $X$ with at least two elements, and define a fuzzy closure operator $c$ by 
				\[
				c(f)= \bigvee_{x_\lambda \leq f} x_1 \quad \text{ for all  } f \in I^X.
				\]  
				If $x_\lambda, k \in I^X$ such that $x_\lambda \leq \mathfrak{Co}(c(k))$, the fuzzy sets $x_1$ and $\underset{y_\lambda \leq f}{\bigvee}{y_1}$ will separate $x_\lambda$ and $k$. Hence, the space $(X,c)$ is \v{C}F-regular. However, it does not satisfy the \v{C}F$T_s$ condition, and therefore it is not a \v{C}F$T_3$ space.
			\end{remark}

			Mashhour and Ghanim\cite{MASH} defined regular fcs differently, as follows:
			
			\begin{Definition} \cite{MASH} \label{thm:mash}
				A fuzzy closure space $(X,c)$ is said to be regular if for all fuzzy points $x_\lambda$ and a fuzzy set $k$ in $X$ such that $x_\lambda \in \operatorname{int}(k)$, there exists a fuzzy set  $f$ in $X$ such that $x_\lambda \in \operatorname{int}(f) \leq c(f) \leq \operatorname{int}(k)$.
			\end{Definition}
			We can discuss the difference between these two definitions. The two definitions of regularity agree on fuzzy topological spaces. The closure operator $c(f)= \underset{x_\lambda \in f}{\bigvee} x_1$ on any set $X$ is a regular fcs on both these definitions.  We can see that Example \ref{thm:nonex} is a non-regular fcs as defined by Mashhur and Ghanim.

			\begin{Theorem}
				If a \fcs\ $(X,c)$ is regular according to Mashhour's definition, then it is a \v{C}F-regular \fcs.
			\end{Theorem}
			
			\begin{proof}
				Let $(X,c)$ be a regular \fcs\ according to Mashhour's definition. Let $x_\lambda$ be a fuzzy point in $X$ and $k \in I^X$ such that $x_\lambda \in \Co(c(k))$. This implies that
				\[
				x_\lambda \in \Co(c(\Co(\Co(k)))) = \intt(\Co(k)).
				\]
				Since $(X,c)$ is regular (by the definition of Mashhour), there exists $f \in I^X$ such that
				\[
				x_\lambda \in \intt(f) \leq c(f) \leq \intt(\Co(k)).
				\]
				Now,
				\[
				c(f) \leq \intt(\Co(k)) = \Co(c(\Co(\Co(k)))) = \Co(c(k)).
				\]
				Therefore,
				\[
				\Co(c(f)) \geq c(k) \geq k.
				\]
				Since $\intt(\Co(f)) = \Co(c(f))$, we obtain $k \leq \intt(\Co(f))$. Thus, $f$ and $\Co(f)$ are neighborhoods that separate $x_\lambda$ and $k$. Hence, $(X,c)$ is \v{C}F-regular.
			\end{proof}

			\begin{remark}
				The property of \v{C}F-regularity is hereditary; that is, any subspace of a \v{C}F-regular fuzzy closure space retains the \v{C}F-regularity property.  
			\end{remark}

			\begin{Theorem}
				\v{C}F-regularity is a fuzzy closure property.
			\end{Theorem}
			
			\begin{proof}
				Let $\theta \colon (X,c) \to (Y,d)$ be a \v{C}F-homeomorphism, and assume that $(X,c)$ is a \v{C}F-regular fcs. Let $y_\gamma$ be a fuzzy point in $Y$ and $k \in I^Y$ such that $y_\gamma \in \Co_Y(d(k))$. Since $\theta$ is a \v{C}F-homeomorphism, $\theta^{-1}(y_\gamma)$ is a fuzzy point in $X$. Now consider the fuzzy point $\theta^{-1}(y_\gamma)$ and the fuzzy set $\theta^{-1}(k) \in I^X$. The condition $y_\gamma \in \Co_Y(d(k))$ implies that
				\[
				\theta^{-1}(y_\gamma) \in \theta^{-1}(\Co_Y(d(k))) = \Co_X(\theta^{-1}(d(k))) = \Co_X(c(\theta^{-1}(k))),
				\]
				where the equalities follow from the properties of \vC F-homeomorphisms.
				
				Since $(X,c)$ is \v{C}F-regular, there exist neighborhoods $f$ and $g$ of $\theta^{-1}(y_\gamma)$ and $\theta^{-1}(k)$, respectively, such that $f \leq \Co_X(g)$. This inequality implies that
				\[
				\theta(f) \leq \theta(\Co_X(g)) = \Co_Y(\theta(g)).
				\]
				By Lemma \ref{thm:interior}, $\theta(f)$ is a neighborhood of $y_\gamma$ and $\theta(g)$ is a neighborhood of $k$. Therefore, $(Y,d)$ is \v{C}F-regular.
			\end{proof}

			\begin{Theorem}
				Let $\mathcal{F} = \{(X_t, c_t) : t \in T\}$ be a family of pairwise disjoint fuzzy closure spaces. The sum $(X, c) = (\underset{t\in T} \bigvee X_t,\oplus c_t)$ is \v{C}F-regular if and only if each $(X_t, c_t)$ is \v{C}F-regular.
			\end{Theorem}
			
			\begin{proof}
				\textbf{($\Rightarrow$)} 
				Assume that the sum $(X, c) = (\underset{t\in T} \bigvee X_t,\oplus c_t)$ is \v{C}F-regular. Let $x_\lambda$ be a fuzzy point in $X_t$ for some $t \in T$, and let $k$ be a fuzzy subset of $X_t$ such that $x_\lambda \in \mathfrak{Co}_{X_t}(c_t(k))$. Viewing $x_\lambda$ and $k$ as fuzzy subsets of $X$ and $x_\lambda \in \mathfrak{Co}_{X_t}(c_t(k))\leq \mathfrak{Co}_{X}(c_t(k))$, the \v{C}F-regularity of $(X, c)$ implies the existence of fuzzy subsets $f, g \in I^X$ that separate $x_\lambda$ and $k$ in $X$. Define $f_t = f \wedge 1_{X_t}$ and $g_t = g \wedge 1_{X_t}$, where $1_{X_t}$ is the characteristic function of $X_t$. These fuzzy subsets $f_t, g_t \in I^{X_t}$ separate $x_\lambda$ and $k$ in $X_t$, demonstrating that $(X_t, c_t)$ is \v{C}F-regular for all $t \in T$.
				
				\textbf{($\Leftarrow$)} 
				Conversely, assume that each $(X_t, c_t)$ is \v{C}F-regular for all $t \in T$. Let $x_\lambda$ be a fuzzy point in $X$ and $k \in I^X$ a fuzzy subset such that $x_\lambda \in \mathfrak{Co}(c(k))$. Suppose $x_\lambda \in X_{t_1}$ for some $t_1 \in T$. Two cases arise:
				
				\begin{enumerate}
					\item{Case 1: $k \wedge 1_{X_{t_1}} = \underline{0}$.}  
					Choose $f = 1_{X_{t_1}}$ and $g = \mathfrak{Co}(1_{X_{t_1}})$. These fuzzy subsets $f, g \in I^X$ separate $x_\lambda$ and $k$ in $X$.
					
					\item{Case 2: $k \wedge 1_{X_{t_1}} \neq \underline{0}$.}  
					Since $(X_{t_1}, c_{t_1})$ is \v{C}F-regular, there exist fuzzy subsets $f_{t_1}, g_{t_1} \in I^{X_{t_1}}$ that separate $x_\lambda$ and $k \wedge 1_{X_{t_1}}$ in $X_{t_1}$. Extend these to $X$ by defining $f = f_{t_1}$ (extended by zero outside $X_{t_1}$) and $g = g_{t_1} \vee \mathfrak{Co}(1_{X_{t_1}})$. These fuzzy subsets $f, g \in I^X$ separate $x_\lambda$ and $k$ in $X$.
				\end{enumerate}
				
				Thus, in both cases, $x_\lambda$ and $k$ can be separated in $(X, c)$, proving that $(X, c)$ is \v{C}F-regular. Therefore, the sum $(X, c)$ is \v{C}F-regular if and only if each $(X_t, c_t)$ is \v{C}F-regular.
			\end{proof}
			
			\section{\v{C}F-normal fuzzy closure spaces}
			\vspace{6pt}
			\begin{Definition}
				A fuzzy closure space $(X,c)$ is said to be \v{C}F-normal if, for all non-empty pair of fuzzy subsets $k_1,k_2\in I^X$ such that $c(k_1) \leq \mathfrak{Co}_X(c(k_2))$, there exists neighborhoods $f_1, f_2 \in I^X$ of $k_1$ and $k_2$, respectively, satisfying $f_1 \leq \mathfrak{Co}_X(f_2)$. A \v{C}F-normal fcs that is also \v{C}$F T_s$ is defined to be \v{C}F$T_4$. 
			\end{Definition}

			As in the case of fuzzy topological spaces, a \v{C}F-normal fuzzy closure space (fcs) is not necessarily \v{C}F-regular. In fact, Example \ref{thm:regular} provides an example of a \v{C}F-normal fcs that fails to be \v{C}F-regular. Moreover, every \v{C}F$T_4$ fcs is clearly \v{C}F$T_3$. Since the indiscrete topology is \v{C}F-normal, it follows that \v{C}F-normality does not necessarily imply the \v{C}F$T_3$ axiom, nor any weaker separation axioms lying below it. Furthermore, if $c$ is a fuzzy topological closure operator, then $(X,\tau(c))$ is a fuzzy normal ($FT_4$) fts if and only if $(X,c)$ is a \v{C}F-normal (\v{C}F$T_4$) fcs.
			
			\begin{remark}
				$(X,\tau(c))$ is a fuzzy normal fts does not imply that $(X,c)$ is a \v{C}F-normal fcs.
				
				Consider the fuzzy closure operator $c$ defined on the fuzzy subsets of $\mathbb{R}$ as follows:
				\[
				c(f) = 
				\begin{cases} 
					\underline{0} & \text{if } f = \underline{0}, \\
					1_{(-\infty, -2)} & \text{if } f \lneq 1_{(-\infty, -2)}, \\
					1_{(2, \infty)} & \text{if } f \lneq 1_{(2, \infty)}, \\
					1_{(-\infty, -1)} & \text{if } f = 1_{(-\infty, -2)}, \\
					1_{(1, \infty)} & \text{if } f = 1_{(2, \infty)}, \\
					\underline{1} & \text{otherwise},
				\end{cases}
				\]
				where $1_A$ denotes the characteristic function of the subset $A \subset \mathbb{R}$, taking the value $1$ on $A$ and $0$ elsewhere.
				
				The fuzzy topology $\tau(c)$ associated with $(\mathbb{R}, c)$ is indiscrete, implying that $(\mathbb{R}, \tau(c))$ is fuzzy normal. The corresponding interior operator on this space is given by
				\[
				{\intt}(f) = 
				\begin{cases} 
					\underline{1} & \text{if } f = \underline{1}, \\
					1_{(-\infty, 1]} & \text{if } 1_{(-\infty, 1]} \leq f \lneq 1_{(-\infty, 2]}, \\
					1_{(-\infty, 2]} & \text{if } 1_{(-\infty, 2]} \leq f \lneq \underline{1}, \\
					1_{[-1, \infty)} & \text{if } 1_{[-1, \infty)} \leq f \lneq 1_{[-2, \infty)}, \\
					1_{[-2, \infty)} & \text{if } 1_{[-2, \infty)} \leq f \lneq \underline{1}, \\
					\underline{0} & \text{otherwise}.
				\end{cases}
				\]
				
				For the fuzzy subsets $1_{(-\infty, -2)}$ and $1_{(2, \infty)}$, it can be observed that $c(1_{(-\infty, -2)}) \leq \mathfrak{Co}(c(1_{(2, \infty)}))$. Since the smallest neighborhoods of $1_{(-\infty, -2)}$ and $1_{(2, \infty)}$ are $1_{(-\infty, 1]}$ and $1_{[-1, \infty)}$ respectively, we can see that these two fuzzy subsets cannot be separated in $(\mathbb{R}, c)$. Hence $(\mathbb{R}, c)$ is not \v{C}F-normal.
			\end{remark}
			
			\begin{example}
				Let $X$ be a non-empty set. Fix an $x\in X$, and define a fuzzy closure operator $c$ as $c(x_\lambda)= x_1$ and if $y\neq x$, $c(y_\lambda)= \mathfrak{Co}(x_1)$, and $c(f)=\underset{z_\rho \leq f}{\bigvee} c(z_\rho)$ for other fuzzy subsets $f$ of $X$. We can see that, this finitely generated fcs $(X,c)$ is \v{C}F-normal.
			\end{example}

			\begin{remark}
				Every finite \v{C}F-regular (\v{C}F$T_3$) fcs is \v{C}F-normal (\v{C}F$T_4$). In general, every finitely generated \v{C}F-regular (\v{C}F$T_3$) fcs is \v{C}F-normal (\v{C}F$T_4$). Moreover, \v{C}F-normality is both a hereditary property and a fuzzy closure property.
			\end{remark}

			\begin{Theorem}
				Let $\mathcal{F}= \{(X_t,c_t): t \in  T\}$ be a family of pairwise disjoint fuzzy closure spaces. The sum $(X,c)= (\underset{t\in T} \bigvee X_t, \oplus c_t)$ is \v{C}F-normal if and only if each $(X_t,c_t)$ is \v{C}F-normal.
			\end{Theorem}
			\begin{proof}
				\par
				$(\Rightarrow)$ Assume that the sum $(\underset{t\in T} \bigvee X_t,\oplus c_t)$ is \v{C}F-normal. For any $t\in T$, let $k_1,k_2$ be two fuzzy subsets of $X_t$ such that $c(k_1) \leq \mathfrak{Co}_{X_t}(c(k_2))$. Extend $k_1$ and $k_2$ to fuzzy subsets of $X$ by defining $k_1(x) = 0$ and $k_2(x) = 0$ for $x \notin X_t$. By the \v{C}ech fuzzy normality of $(X, c)$, there exist fuzzy subsets $f, g \in I^X$ that separate $k_1$ and $k_2$ in $X$. The restrictions $f \wedge 1_{X_t}$ and $g \wedge 1_{X_t}$ in $I^{X_t}$ then separate $k_1$ and $k_2$ in $X_t$. Hence, $(X_t, c_t)$ is \v{C}F-normal for each $t \in T$.
				\par 
				$(\Leftarrow)$ Conversely suppose that each $(X_t,c_t)$ is \v{C}F-normal for all $t \in T$. Let $k_1,\text{ and } k_2$ be fuzzy subsets of $X$ such that $\oplus c_{t}(k_1)\leq \mathfrak{Co}_X (\oplus c_{t}(k_2))$. For each $t \in T$, the restrictions $k_1 \wedge 1_{X_t}$ and $k_2 \wedge 1_{X_t}$ belong to $I^{X_t}$, and the condition $c_t(k_1 \wedge 1_{X_t}) \leq \mathfrak{Co}_{X_t}(c_t(k_2 \wedge 1_{X_t}))$ holds. By the \v{C}F-normality of $(X_t, c_t)$, there exist fuzzy subsets $f_t, g_t \in I^{X_t}$ that separate $k_1 \wedge 1_{X_t}$ and $k_2 \wedge 1_{X_t}$. Define $f =  \underset{t\in T} \bigvee f_t$ and $g =  \underset{t\in T} \bigvee g_t$ in $I^X$. These fuzzy subsets $f$ and $g$ separate $k_1$ and $k_2$ in $X$, proving that $(X, c)$ is \v{C}F-normal.
			\end{proof}

			\section{Conclusions and Future Research}
			Fuzzy topological spaces do not constitute a natural boundary for the validity of theorems, but many theorems can be extended to what are called fuzzy closure spaces. In this study, we defined interior operator in fuzzy closure spaces, as well as neighborhood systems and found an equivalent condition for \v{C}F-continuity. We established a number of separation axioms, including \v{C}F$T_0$,\v{C}F$T_1$,\v{C}F$T_s$,\v{C}F$T_2$,\v{C}F$T_3$,\v{C}F$T_4$, \v{C}F-Urysohn, \v{C}F-regular, and \v{C}F-normal, using the neighborhood system. We have discovered certain characteristics of these separation axioms. Here are some questions that need more attention. Productive behaviour of some separation axioms are discovered, others yet to be investigated. We have found a characterization theorem for \v{C}F$T_1$ fcs. Researchers can investigate the characterization theorem for other separation axioms.
			Identify the lattice of \v{C}F$T_0$, \v{C}F$T_1$, \v{C}F$T_2,$ etc fuzzy closure operators. Extend all these concepts to $L$-closure spaces.

		{\fontsize{16}{24}\section{Acknowledgments}}
		First author is supported by the Senior Research Fellowship
		of CSIR (Council of Scientific and Industrial Research, India).
		
		{\fontsize{16}{24}\section{Author Contribution}}
		Every author has reviewed and approved the published version of the work.
		{\fontsize{16}{24}\section{Funding}}
		The authors declare that no external funding or support was received for the research presented in this paper, including administrative, technical, or in-kind contributions.
		{\fontsize{16}{24}\section{Data Availability}}
		All data supporting the reported findings in this research paper are provided within the manuscript.
		
		{\fontsize{16}{24}\section{Conflicts of Interest}}
		According to the authors, there is no conflict of interest with regard to the reported research findings. There was no role from funders in the design of the study, data collection, analysis, interpretation, article preparation, or decision to publish the findings.

\end{document}